\documentclass[12pt,final]{article}   %%%%%

\usepackage{amsmath}
\usepackage{amsthm}
\usepackage{amssymb}
\usepackage{amsfonts}
\usepackage{latexsym}               % LaTeX Symbol Font,amssymb
\usepackage{amsgen}
\usepackage[mathscr]{eucal}
\usepackage{mathrsfs}
\usepackage{enumerate}
\usepackage{bm}
\usepackage{here}

\usepackage{mathptmx} % other option is txfonts
\usepackage{color}
\usepackage{fancybox}
\usepackage{graphicx}
\usepackage{cite}

\usepackage{showkeys}

\newtheorem{lem}{Lemma}[section]
\newtheorem{thm}[lem]{Theorem}
\newtheorem{pro}[lem]{Proposition}

\newtheorem{defn}[lem]{Definition}

\newcommand{\bqn}{\begin{equation}}
\newcommand{\eqn}{\end{equation}}
\newcommand{\beqx}{\begin{equation*}}
\newcommand{\eeqx}{\end{equation*}}
\newcommand{\barr}{\begin{array}}
\newcommand{\earr}{\end{array}}
\newcommand{\beqn}{\begin{eqnarray}}
\newcommand{\eeqn}{\end{eqnarray}}
\newcommand{\beqnx}{\begin{eqnarray*}}
\newcommand{\eeqnx}{\end{eqnarray*}}
\newcommand{\bmt}{\begin{multline}}
\newcommand{\emt}{\end{multline}}

\numberwithin{equation}{section}

\newcommand{\pa}{\partial}

\def\dim{\mathop{\rm dim}\nolimits}

\newcommand{\sF}{\mathscr{F}}

\newcommand{\sK}{{\mathscr K}}
\newcommand{\sL}{\mathscr{L}}
\newcommand{\sM}{\mathscr{M}}

\newcommand{\sO}{\mathscr{O}}

\newcommand{\sS}{\mathscr{S}}

\newcommand{\sW}{\mathscr{W}}

\newcommand{\cA}{{\mathcal A}}

\newcommand{\cC}{{\mathcal C}}

\newcommand{\cF}{{\mathcal F}}

\newcommand{\cK}{{\mathcal K}}
\newcommand{\cL}{{\mathcal L}}

\newcommand{\cS}{{\mathcal S}}
\newcommand{\cO}{{\mathcal O}}

\newcommand{\bbn}{{\mathbb N}}

\newcommand{\real}{{\mathbb R}}

\newcommand{\al}{\alpha}

\newcommand{\ga}{{\gamma}}

\newcommand{\ve}{\varepsilon}

\newcommand{\la}{\lambda}

\newcommand{\ps}{\psi}

\newcommand{\ro}{\rho}

\newcommand{\er}{\eqref}
\newcommand{\lb}{\label}
\newcommand{\qu}{\quad}

\title{High-Voltage Ionized Gas \\ 
with Spherical Cathode Emission}

\author{%
{\large\sc Walter A. Strauss${}^1$}
{\normalsize and}
{\large\sc Masahiro Suzuki${}^2$}
}

\date{%
\normalsize
${}^1$%
Department of Mathematics and Lefschetz Center for Dynamical Systems, 
Brown University,   
\\
Providence, RI 02912, USA
\\ [7pt]
${}^2$%
Department of Computer Science and Engineering, 
Nagoya Institute of Technology,
\\
Gokiso-cho, Showa-ku, Nagoya, 466-8555, Japan
}

\begin{document}

\maketitle

\begin{abstract}
 We consider a plasma that is created by a high voltage difference, which is known as a Townsend gas discharge.   The plasma is confined to the region between two concentric spheres, one of which is a cathode and the other an anode.  %or else  between two concentric cylinders.     
 Ion-electron pairs are created by collisions inside the plasma.  Additional electrons 
 enter the plasma by collisions of ions with the cathode.
 We prove under certain conditions that there are many steady states exhibiting gas discharge, beginning with a `sparking' voltage.  
 In fact, there is an analytic one-parameter family of them that connects the non-ionized gas to a plasma %either with a sparking voltage $\lambda^\dagger$ or 
 with arbitrarily high ionization or arbitrarily high potential, or else the family ends at an `anti-sparking' voltage.
\end{abstract}

\begin{description}

\item[{\it Keywords:}]  
Ionization, Morrow model, hyperbolic-parabolic-elliptic system, global bifurcation

\item[{\it 2020 Mathematics Subject Classification:}]
35M33, %Initial-boundary value problems for mixed-type systems of PDEs
35B32, %Bifurcations in context of PDEs
35B35, %Stability in context of PDEs
76X05 %Ionized gas flow in electromagnetic fields; plasmic flow

\end{description}

\section{Introduction} 
\subsection{Ionization model} 
We consider the ionization of a gas  
due to a strong applied electric field.  The high voltage thereby creates a plasma  
and may induce very hot electrical arcs.  
Over a century ago Townsend experimented with a pair of parallel plates 
to which he applied a strong voltage that produces cascades of free electrons and ions.  
This phenomenon is called the Townsend discharge or avalanche. 
The minimum voltage for this to occur is called the sparking voltage.  
There are two main sources of ionization.     
The result of collisions of gas particles within the plasma is called the $\alpha$-mechanism.   The emission of electrons caused by the impacts of the ions with the cathode is called the $\gamma$-mechanism.
For more details of the physical phenomena, we refer the reader to \cite{Rai}.

In this paper the gas occupies the region $\Omega$ between two concentric spheres.   The inner sphere may be the anode $\cA$ and the outer one the cathode $\cC$, or vice-versa. Within the plasma $\Omega$ the model is as follows.  
Let $\rho_i, u_i, \rho_e, u_e$ and $-\Phi$ denote the the ion density, ion velocity, electron density, electron velocity and electrostatic potential, respectively.  
%Of course, $\rho_i$ and $\rho_e$ must be nonnegative.  
Within $\Omega$ they satisfy the 
								equations.  
\begin{subequations}\lb{Mmodel}
\begin{gather}
\partial_t\rho_i + \nabla\cdot(\rho_i {\bm u}_i)  =  k_e h(|\nabla\Phi|) \rho_e ,
\lb{eqi} \\
\partial_t\rho_e + \nabla\cdot(\rho_e {\bm u}_e)  =  k_e h(|\nabla\Phi|) \rho_e ,
\lb{eqe} \\
\Delta \Phi =\ro_i-\ro_e, 
\lb{eqp} \\
{\bm u}_i := k_i\nabla\Phi, \quad
{\bm u}_e:=-k_e \nabla\Phi-k_e \frac{\nabla \rho_e}{\rho_e}, 
\lb{equ}
\end{gather}
where 
\begin{gather}
h(\ell) := a\ell e^{-{b}/{\ell}}. 
\label{gh}
\end{gather}
The prescribed initial and boundary conditions are 
\begin{gather}
  (\ro_i,\ro_e)(0,x)=(\ro_{i0},\ro_{e0})(x),  
    \quad x\in \Omega,
  \lb{i1} \\
  \ro_i(t,x)=\ro_e(t,x)=\Phi(t,x)=0, \quad t>0, \ x \in \cA,
  \lb{ba} \\
  (\ro_e\bm{u}_{e}\cdot\bm{n})(t,x)=-\gamma (\ro_i\bm{u}_{i} \cdot\bm{n} )(t,x), \quad \Phi(t,x)=\lambda>0, \quad  t>0, \ x \in \cC.
  \lb{bc}
\end{gather}
\end{subequations}

The properties of this model depend on the six positive constants $\la, a,b,\ga,k_e,k_i$.  
Of course we assume $\rho_i\ge0$ and $\rho_e\ge0$.  
The purpose of this article is to prove that there are many solutions that are both radial and independent of time.  
The first two equations express the transport of ions and electrons, while their 
right sides express the rate per unit volume of ion--electron pairs created 
by the impacts of the electrons within $\Omega$.    
This is called the $\alpha$-mechanism.  
Specifically, the coefficient $a\exp(-b |\nabla\Phi |^{-1} )$ is 
the first Townsend ionization coefficient $\al$, 
which was determined empirically as described in 
Section 4.1 of \cite{Rai}.  It can also be found explicitly in equation (A1) of \cite{Mo1}, for instance.  
Equation \er{eqp} states that the divergence of the electric field equals the net charge.  
The equations \eqref{equ} express the impulse of the electric field ${\bf E}=-\nabla \Phi $ 
on the particles in opposite directions.  
The last term in \er{equ}  expresses the motion of the electrons due to the gradient of their density.   
However, the heavy ions barely move when impacted by electrons.  
Substituting the constitutive velocity relations \er{equ} 
into the continuity equations \er{eqi} and \er{eqe}, 
we see that the system is of hyperbolic-parabolic-elliptic type. 

As expressed in \er{ba}, the boundary conditions at the anode are due to the assumption that it is a perfect conductor.  
The cathode is negatively charged, producing a voltage difference $\lambda$.  
The secondary emission at the cathode (or so-called $\ga$-mechanism) 
is described by the first equation in \er{bc}, 
%\bqn   \label{gamma}
%\rho_eu_e=-\ga\rho_iu_i   \eqn
where $\ga > 0$ is average number of electrons ejected from the cathode by an ion impact.

The results of this paper are also valid for the region between two concentric cylinders if we study steady states that are constant in the direction of their central axis.  In that situation our treatment effectively deals only with the cross section $\Omega$, which is an annulus in the plane perpendicular to the axis. The analysis of the cylindrical case differs in only relatively minor technical ways from the spherical case.  For the sake of simplicity in this paper we only treat the spherical case.

\subsection{Physical and mathematical background}  
Townsend introduced the $\ga$-mechanism in order to guarantee that gas discharge occurs beyond a sparking voltage.  
Since Townsend's work many models have been proposed to describe this ionization phenomenon \cite{AB,DL1,DW,DT1,Ku1,Ku2,LRE,Mo1}.  
In 1985 Morrow \cite{Mo1} was probably the first to provide a realistic model of its detailed mechanism. 
For simplicity we limit ourselves to a single species of positive ions and   
we do not consider various other mechanisms such as `attachment' or `recombination'.  
Morrow's model is physically reliable.
Indeed, the article \cite{DL1} by Degond and Lucquin-Desreux derives the model 
from the general Euler-Maxwell system by scaling assumptions, 
in particular by assuming a very small mass ratio between the electrons and ions. 
In an appropriate limit the Morrow model is obtained 
at the end of their paper in equations (160) and (164), which we have specialized to 
assume constant temperature and no neutral particles.  See also Chapter 5 in \cite{LL}.   
 
Townsend defined the {\it sparking voltage} as the threshold of voltage at which gas discharge occurs and continues.  
Mathematically, we will define the sparking voltage $\lambda^*$  as the lowest voltage for which a certain natural linear operator \er{eqU} has a non-trivial nullspace.  
In the presence of a $\ga$-mechanism the boundary condition for \er{eqU} is 
{\it non-local} with parameters $\ga$ and $\la$.  
%{\blue As a result, $A$ becomes non-self-ajoint.}

%So far as we know, there is no rigorous mathematical study 
%that analyzes gas discharge in domains other than plates.  
 There are many physical and mathematical studies of gas discharges with different configurations of anode and cathode but without any analysis of the fundamental partial differential equations.
The article \cite{LU1} investigated the gas discharges on cylinders or so-called {\it wire to wire} domains that are regions exterior to a pair of parallel cylinders. 
%The analysis is similar in spirit to Townsend's theory.  On the other hand, 
Durbin and Turyn \cite{DT1} numerically  simulated non-trivial steady states of ionization in coaxial concentric and accentric cylinders 
by using a similar model to Morrow's.   
Morrow \cite{Mo2} also simulated the ionization in concentric spheres using his model.
We refer the reader to \cite{TK1}, which reviews the recent progress of numerical simulations for gas ionization.

Suzuki and Tani in \cite{ST1} gave the first mathematical analysis of the Morrow model with both the $\alpha$ and $\gamma$-mechanisms in which   
they proved the time-local solvability.   
In \cite{ST2} without the $\ga$-mechanism they introduced the concept of a sparking voltage and proved the stability and instability of trivial solutions.  

Subsequently we wrote three articles that proved the global bifurcation of steady state solutions in gas ionization under various assumptions.  
%In all of them ionization occurs due to the $\alpha$-mechanism.  
\cite{SS1} and \cite{SS2} treated the  simpler geometry of a gas between two parallel plates and solutions that are constant on each parallel plane.  On the other hand, \cite{SS3} considered the region between two concentric spheres and also, to some extent, the region between two star-shaped surfaces.  Because steady states that  depend only on the radial variable between concentric spheres satisfy differential equations with variable coefficients, the solutions are much less explicit than in the case of parallel plates.  
%However $\ga=0$ in \cite{SS3}. 

The $\ga$-mechanism is a much more subtle phenomenon.  In the three previous articles, only \cite{SS2} considered it.  In that article the nullspace of the linearized problem is quite explicit because it satisfies an ODE with constant coefficients due to the simple geometry.  
The present paper treats the $\ga$-mechanism in  a spherical (or cylindrical) shell so that the coefficients are variable and the nullspace is not explicit.  

As already mentioned, the system of equations for the steady states is mixed hyperbolic-elliptic.  It is hyperbolic in $\rho_i$ and elliptic in $\rho_e$ and $\Phi$.    However, all known proofs of global bifurcation essentially require ellipticity.  What we are able to do is take advantage of the steady states being functions of one variable only, by substituting $\rho_i$ into the elliptic equations.  
Here we are also able to introduce a precise definition \er{spv1} of the sparking voltage in a general context. 
In contrast to Townsend's original theory, a sparking voltage may or may not exist, both with and without a $\ga$-mechanism.  
The variable coefficients (due to the geometry) together with the non-local boundary condition (due to the $\ga$-mechanism) introduce major complications to our analysis.  

% **********************************************************
\section{Main results}
In this section we state our main theorem on global bifurcation. 
We define $H$ as the harmonic function in $\Omega$ %the solution to the Laplace equation $\Delta H=0$
 with the boundary conditions $H=0$ on $\cA$ and $H=1$ on $\cC$. 
The problem \eqref{Mmodel} has the trivial stationary solution $(\rho_{i},\rho_{e},\Phi)=(0,0,\lambda H)$ for any $\la$.
For simplicity in what follows, we let
\begin{gather*}%\label{para1}
 r_{1}=1, \quad r_{2}=2, \quad I=(1,2), \quad \cA=\{r=1\}, \quad \cC=\{r=2\},  
\end{gather*} 
where $r$ is the radial coordinate.  
In this case the harmonic function $H$ is exactly 
\begin{gather}\label{H1}
H=2\left(1-\frac{1}{r}\right) \quad \text{ so that }\ \  \partial_{r} H =\frac{2}{r^{2}}.
\end{gather}
We also define 
\begin{gather*}    %\label{def:g}
g(\ell) := h(\ell)-\frac{\ell^2}4 = a\ell e^{-b/\ell} - \frac{\ell^2}4  \quad\text{ so that }\ \  
g\left(\la \partial_{r} H \right)  = a\frac{2\la}{r^2} e^{-\frac{br^2}{2\la}} - \frac{\la^2}{r^4}.  
\end{gather*}

It will turn out that a suitable approximation of the pair of equations \er{eqe}, \er{equ}  after linearization  is the following {\it electron system}:  
\begin{subequations}\label{eqU}
\begin{gather}  \label{eqUa} 
-\frac{1}{r^{2}} \partial_{r} ( r^{2} \partial_{r} \varphi_{e} ) - g(\lambda \partial_{r} H) \varphi_{e}=0,
\\
\varphi_{e}(1)=0,\   \text{ and }\  \varphi_{e}(r)>0\ \ \text{ in }\ I, \label{eqUb}\\
B(\lambda,\varphi_{e}):=\partial_{r} \varphi_{e}(2) + \frac{\lambda}{4} \varphi_{e}(2) 
- a \gamma  \frac{\la}{2} e^{-\frac{\la}{2}} \int_{1}^{2}  
e^{-\frac{b r^{2}}{2 \la}}e^{\frac{\la}{r}} \varphi_{e}(r) \,dr=0.  \label{eqUc}
\end{gather}
\end{subequations} 
Because \eqref{eqUa} is a linear ODE and $u(1)=0$, 
its set of solutions is at most one-dimensional.  
Moreover, only for special values of the parameters could \er{eqUc} also be satisfied. 
It is particularly nontrivial to find such a special value, because the operator $-\Delta-g(\cdot)$ together with the `boundary' conditions $\varphi_{e}(1)=0$ and $B(\la,\varphi_{e})=0$ is not self-adjoint.
					\begin{defn}  
					The set of parameters $a,b,\ga$ 
for which there is a sparking voltage is defined as 
\begin{gather*} 
\sS:=\left\{(a, b, \gamma) \, | \, 
\text{$\exists$ pair $(\lambda,\varphi_{e}) \in \mathbb R_{+} \times H^2(I)$ that solves the electron system  \eqref{eqU}} \right\}.
\end{gather*}  
\end{defn}
We call $\sS$ the {\it sparking parameter set}.  
For each $(a,b,\ga)\in \cS$,  we define its {\it sparking voltage} $\lambda^{\dagger}$ by  
\begin{gather}
\la^\dagger := \inf\left\{\la>0\ |\ 
\exists\  \varphi_{e}\in H^2(I) \text{ so that } (\la, \varphi_{e}) \text{ solves  \eqref{eqU}}\right\}.
\label{spv1}
\end{gather}
Thus $0\le\la^\dagger<\infty$.  
{\it We will say that the sparking voltage exists if $(a,b,\ga)\in \cS$.  }

				\begin{lem}
The sparking voltage $\la^\dagger$ is positive for every $(a,b,\ga)\in \cS$.  
\end{lem}
				\begin{proof}   
Suppose $0< \la \leq 1$.
We multiply \eqref{eqUa} by $r^2 \varphi_{e}$, integrate.  
Using also $\varphi_{e}(1)=0$ and \eqref{eqUc}, we  obtain 
\begin{align*}   
\|\varphi_{e}'\|^2_{L^2} & \le C_1 \la (|\varphi_{e}(2)|+\|\varphi_{e}\|_{L^{2}})|\varphi_{e}(2)| + C_2 \la \|\varphi_{e}\|^2_{L^2} \le C_3\la \|\varphi_{e} '\|_{L^2}^{2},
\end{align*}
where we have also used Sobolev's and Poincar\'e's inequalities.
Clearly, if $\la$ sufficiently small, we deduce $\|\varphi_{e}'\|^2_{L^2}=0$.
Since $\varphi_{e}(1)=0$, we have $\varphi_{e} \equiv 0$. 
Thus \eqref{eqU} has no solution for small enough $\la$, so that $\la^\dagger>0$.  
\end{proof}

Assuming that the sparking voltage does exist 
and the transversality condition \eqref{TC1} given below does hold,
we will prove that {\it there exist many other steady solutions}, in fact a whole global curve of them. 

\begin{defn}
We define two Hilbert spaces of radial functions in the interval $I:=(r_{1},r_{2})$ as  
\begin{align}    
   X := &  \{(f_1,f_2,f_3) \in H^1(I)\times H^2(I)\times H^3(I) \ \big|  \ \ 
f_1|_\cA=  f_2|_\cA=  f_3|_\cA=f_3|_\cC=0  \},    \label{spaceX}
%\\
%X: \ &\rho_i \in \{f\in H^{1}(I);f|_{\cA}=0\}, \ \ R_e \in \{f\in H^2(I);f|_{\cA}=0\}, \ \
%V \in \{f\in H^3(I);f|_{\cC}=f|_{\cA}=0 \};
\\
Y:= & L^{2}(I) \times L^{2}(I) \times H^{1}(I) \times\mathbb R,  \label{spaceY}
\end{align}
where $L^{2}(I)$ and $H^{k}(I)$ denote the Lebesgue and Sobolev spaces equipped the following inner products with the {\it weight } $r^{2}$ 
(naturally coming from the reduction of $x\in\real^3$ to $r\in\real$):  
\begin{gather}\label{inner-product1}
\langle u,v \rangle:= \int_{r_{1}}^{r_{2}} r^{2} uv \, dr, \quad 
\langle u,v \rangle_{H^{k}}:= \int_{r_{1}}^{r_{2}} r^{2} uv \, dr+ \sum_{j=1}^{k}\int_{r_{1}}^{r_{2}} r^{2} (\partial_{r}^{j}u) (\partial_{r}^{j} v) \, dr.
\end{gather}
\end{defn}

\begin{thm} \label{mainthm0}
Assume that the sparking voltage $\lambda^\dagger$ exists, and the transversality condition \eqref{TC1} holds at the sparking voltage $\lambda^{\dagger}$.
Then there exists a unique continuous one-parameter family $\mathcal{K}$ (an analytic  curve) 
of radial steady solutions $(\rho_{i},\rho_{e},\Phi) \in X$ of the problem \eqref{Mmodel} with the following properties.  
Both densities are positive in $\Omega$, %$\rho_i\in H^1(I), \ \rho_e\in H^2(I), \ \Phi\in H^3(I)$, 
the curve begins at the trivial solution $(0,0,\lambda^{\dagger} H)$ with voltage $\lambda^\dagger$ 
and it `ends' with one of the following three alternatives: 

Either (i) the density $\sup_{r \in I} (\rho_i + \rho_e)(r)$ becomes unbounded along $\sK$, 

Or (ii) the potential $\sup_{r \in I} \Phi(r)$ becomes unbounded along $\sK$, 

Or (iii) the curve ends at a different trivial solution with some voltage $\lambda^\ddagger > \lambda^\dagger$.
\end{thm}

The {\it transversality condition} is the standard one that appears in local bifurcation theory.  
We will show in Section \ref{LinearPro} that this standard condition 
can be written explicitly as 
\begin{gather}\label{TC1}
F(a,b,\gamma )\neq 0,
\end{gather}
where
\begin{align}    
F(a,b,\gamma)&:=-{\gamma} e^{\frac{\la}{2}}
\int_1^2 r^{2} \psi_e(2) \left\{ h'\left( \lambda \partial_{r} H \right) (\partial_{r} H) - h\left( \lambda \partial_{r} H \right)\frac{H}{2} \right\}e^{-\frac{\lambda }{2}H}\varphi_e \,dr 
   \notag\\
&\quad -\int_1^2 r^{2} \psi_{e} g'(\lambda \partial_{r} H) (\partial_{r} H) \varphi_e \, dr+\psi_e(2)          
\left\{\varphi_e(2) - 2\partial_r\varphi_e (2)  - \frac{\lambda}  {2}\varphi_e(2)\right\}.
\label{transversality0}
\end{align}
Here $\varphi_{e}$ is a solution of the electron system \eqref{eqU}, and 
$\psi_{e}$ is a solution of the following boundary value problem:
\begin{subequations}              \label{eqW}
\begin{gather}
-\tfrac{1}{r^{2}} \partial_{r} ( r^{2} \partial_{r} \psi_e )-g(\lambda \partial_{r} H)\psi_e
-h(\lambda\partial_{r} H)e^{-\frac{\lambda}{2}H} {\gamma} e^{\frac{\la}{2}}\psi_e(2)=0,
\label{PhiEq10}\\
\psi_e(1)=\partial_{r} \psi_e(2) + \tfrac{\lambda}{4} \psi_e(2)=0.
%- \gamma \frac{\lambda}{2} e^{-\frac{\lambda}{2}} \int_{1}^{2} a e^{-\frac{b r^{2}}{2 \lambda}}e^{\frac{\lambda}{r}}S_e(r) \,dr=0, 
\end{gather}
\end{subequations}
In Section \ref{LinearPro} we will also show that \eqref{eqW} is the adjoint problem of \eqref{eqU} 
and that the set of solutions of \eqref{eqW} is at most one-dimensional.  

%%%%%%%%%%%%%%%%%%%%%%%%%%%%%%%%%
%%%%%%%%%%%%%%%%%%%%%%%%%%%%%%%%%

Of course, it is important to know that the assumptions (existence of a sparking voltage and transversality)
in Theorem \ref {mainthm0} are indeed satisfied for a significant set of parameters.  
The following theorem provides such a set of parameters with large enough cathode emission.  

\begin{thm}\label{SparkingVoltage}
The sparking parameter set $\cS$ contains the set $\Gamma$ %dominated by cathode emission,  
   defined as
\begin{gather*}%\label{SetA}
\Gamma :=\left\{ (a, b, \gamma) \, ; \,  \gamma>\frac 1a,\ \ b>\frac 4ea\right\}. %  \subset \cS.
\end{gather*}
Furthermore, for almost every $(a, b, \gamma)\in \Gamma$, the transversality condition \eqref{TC1} holds at the sparking voltage $\lambda^{\dagger}=\lambda^{\dagger}(a,b,\gamma)$.
\end{thm}

This paper is organized as follows.
In Section 3 we study the linearized problem to be able to apply the global bifurcation theorem in \cite{CS1}.
Specifically, in subsections \ref{S3.1} and \ref{S3.2}, we investigate the nullspace and the range of the linearized operator $\cL$ making use of the adjoint operator. 
In particular,  the electron system \eqref{eqU} plays an essential role in determining that the linearized operator has zero eigenvalue.
In subsection \ref{S5}, we formulate the transversality condition as shown in \eqref{TC1}.
%We prove in Lemma \ref{trans1} that this condition is valid for almost every $(a,b,\gamma)$.  
Then in Section 4 we apply the global bifurcation technique to construct a {\it global} curve $\sK$ of steady solutions $(\rho_i,\rho_e, \Phi)$.  
The ellipticity of the system plays a crucial role.
The general properties of this global curve are summarized in Theorem \ref{Global1}. 
The curve may include mathematical solutions with positive densities as well as solutions with negative ``densities''.  
In Section \ref{SPD} we naturally restrict our attention to positive densities.  
Further analysis of the possible ways that the curve may ``terminate'' is then  provided.  
The main conclusion (outlined  in Theorem \ref{mainthm0}) is given in Theorem  \ref{Global2}.  
In case the voltage becomes unbounded, it is proven in Section \ref{SAHV} that the densities tend to zero.  
Finally, Section \ref{SVDCE} provides the proof of Theorem \ref{SparkingVoltage}, which exhibits a significant set of parameters for which both sparking voltage and transversality conditions are verified.

We briefly discuss the analytical challenges encountered.
The authors' previous works \cite{SS1} and \cite{SS2} only treated
parallel plates, for which the stationary states satisfy an ordinary differential equation with constant coefficients.
In contrast, the present paper addresses ionization within spherical shells and cylindrical shells. 
Consequently, the linearized operator now possesses variable coefficients, so that its nullspace is no longer explicit.
In our other previous work \cite{SS3}, we also investigated spherical and cylindrical shells without considering secondary emission at the cathode. In that case, the boundary condition \eqref{eqUc} in the electron system \eqref{eqU} is replaced by $\varphi_{e}(2) = 0$, rendering the operator $-\Delta - g(\cdot)$ self-adjoint under the conditions $\varphi_{e}(1) = \varphi_{e}(2) = 0$. This self-adjointness allows for the application of the Rayleigh quotient to determine that the linearized operator has a zero eigenvalue.
However, as noted above, when the boundary conditions are $\varphi_{e}(1) = 0$ and $B(\lambda, \varphi_{e}) = 0$, the operator $-\Delta - g(\cdot)$ becomes non-self-adjoint. As a result, %the Rayleigh quotient is no longer applicable, necessitating   
alternative analytical approaches are required for the eigenvalue analysis.
Furthermore, Theorem \ref{SparkingVoltage}, which is based on on the cathode emission, is delicate and entirely new.

%%%%%%%%%%%%%%%%%%%%%%%%%%%%%%

\section{Linearized problem}\label{LinearPro}  

It is convenient to rewrite the initial--boundary value problem \er{Mmodel} 
in terms of the %new unknown functions 
modified functions $(\rho_{i },R_{e},V)$, where
\begin{gather}\label{newf0}
R_e:=\ro_e e^{\frac{\lambda}{2}H}.
\end{gather}
We decompose the electrostatic potential as
\bqn   \label{Phi} 
 \Phi=V+\lambda H  %=V+\lambda \left(2- \frac{2}{r}\right),
\eqn 
where $V$ solves $\Delta V = \rho_i-\rho_e$ and vanishes on $\partial\Omega$.

We rewrite the stationary system ($\pa_t=0$)  in three dimensions as follows.  
We substitute \er{newf0} and \er{Phi} into the system \er{Mmodel}.  
We also substitute \er{equ} into \er{eqi} and \er{eqe}.  This brings us to the system 
\begin{equation}\lb{sp0}
\sF_j(\lambda,\rho_i,R_e,V)=0 \qu \text{for $j=1,2,3,4$,}  
\end{equation}
 where
\begin{align*}
\sF_1(\lambda,\rho_i,R_e,V)
&:=\frac{k_i}{r^{2}} \partial_{r} \left\{ r^{2} \rho_i  \left(\partial_{r} V+ \la\partial_{r} H \right) \right\}
-k_eh\left(\left| \partial_{r} V+ \la\partial_{r} H \right|\right)e^{-\frac{\lambda}{2}H}R_e,
			\\
\sF_2(\lambda,\rho_i,R_e,V)
&:=-\frac{1}{r^{2}} \partial_{r} ( r^{2} \partial_{r} R_e )
-\partial_{r} V \partial_{r} R_{e}
			\\
&\quad +\left\{\frac{\lambda}{2} \partial_{r} V \partial_{r} H
-\frac{1}{r^{2}} \partial_{r} ( r^{2} \partial_{r} V)
+\frac{\lambda^2}{4}|\partial_{r} H|^{2}
-h\left(\left| \partial_{r} V+ \la\partial_{r} H \right|\right)\right\}R_e,
			\\
\sF_3(\lambda,\rho_i,R_e,V)
&:=\frac{1}{r^{2}} \partial_{r} ( r^{2} \partial_{r}  V) - \rho_i + e^{-\frac{\lambda}{2}H}R_e,
			\\
\sF_4(\lambda,\rho_i,R_e,V)&:= \left\{ \left. \partial_{r} R_{e} +\left(\frac{\lambda}{2}\partial_{r} H + \partial_{r} V \right) R_{e} - \gamma \frac{k_{i}}{k_{e}} e^{\frac{\lambda}{2}H} \left(\partial_{r} V + {\lambda}\partial_{r} H \right) \rho_{i } \right\}\right|_{\cC}.  
\end{align*}

Some of the boundary conditions are encoded in the space $X$.  
It is easy to see that this stationary system has the trivial solution
$(\rho_i,R_e,V)=(0,0,0)$ for any $\lambda>0$, which corresponds to $(\rho_i,\rho_e,\Phi)=(0,0,\lambda H)$.
It is also clear that $\cF := (\cF_1,\cF_2, \cF_3,\cF_4)$ maps $(0,\infty)\times X$ into $Y$ 
continuously and that $\cF$ is quadratically bounded:
\begin{gather*}
\| \cF(\la,\rho_i,R_e,V)\|_Y   \le   C(\la)  \{ 1 + \|V\|_{H^3}\}  
\{1 + \|\rho_i\|_{H^1}  +  \|R_e\|_{H^2} \}, 
\end{gather*}
where $C(\la)$ is a constant that depends only on $\la$.

In order to prove bifurcation, we must study the linearized problem at the bifurcation point, which in our case is the trivial solution.  
The bifurcation parameter is $\la$.
For any $\la>0$, let $\sL= \partial_{(\rho_i,R_e,V)}\sF(\lambda,0,0,0)$ denote the linearized operator around the trivial solution $(\rho_i,R_e,V)=(0,0,0)$.
%where $\lambda^{\dagger}$ is the sparking voltage in \eqref{sparkingV}. 
We regard $\sL$ as an unbounded operator on 
$L^2\times L^2\times L^2$ with its domain $D(\sL) = X \subset  H^1\times H^2\times H^3$ 
and with its range in $Y$.  The spaces $X$ and $Y$ were defined in \eqref{spaceX} and \eqref{spaceY}.
From \eqref{sp0}   it is clear that $\sL=(\sL_{1},\sL_{2},\sL_{3},\sL_4)$, where
\begin{subequations}\label{LinearEq0}
\begin{align}
\sL_{1} (S_i,S_e)         %=\partial_{(R_i,R_e,V)} \sF_{1}(\lambda,0,0,0)
&:=  2\frac{k_i}{r^{2}} \lambda \partial_{r} S_i
-k_eh\left( \lambda \partial_{r} H \right)e^{-\frac{\lambda}{2}H}S_e,
\label{LinearEq1} \\
\sL_{2} (S_e)         %=\partial_{(R_i,R_e,V)} \sF_{2}(\lambda,0,0,0)
&:=-\frac{1}{r^{2}} \partial_{r} ( r^{2} \partial_{r} S_e ) - g(\lambda \partial_{r} H)S_e,
\label{LinearEq2} \\
\sL_{3} (S_i,S_e,W)        %=\partial_{(R_i,R_e,V)} \sF_{3}(\lambda,0,0,0)
&:=\frac{1}{r^{2}} \partial_{r} ( r^{2} \partial_{r} W) - S_i + e^{-\frac{\lambda}{2}H}S_e,
\label{LinearEq3} \\
\sL_4(S_i,S_e)&:= \partial_{r} S_{e}(2) + \frac{\lambda}{4} S_{e}(2) - \gamma \frac{k_{i}}{k_{e}} e^{\frac{\lambda}{2}} \frac{\lambda}{2} S_{i }(2).
\label{LinearEq4}
\end{align}
\end{subequations}
%Here $g(s):=ase^{-\frac{b}{s}}-{s^{2}}/{4}$.
In addition to \eqref{LinearEq4}, the boundary conditions are as follows:
\begin{gather}\label{BC0}
S_{i}(1)=S_{e}(1)=W(1)=W(2)=0.
\end{gather}

Let us consider the nullspace $N(\cL) \subset X$.  
Integrating the first equation $\sL_{1} (S_i,S_e)=0$ together with $S_{i }(1)=0$, we see  that  the fourth equation $\sL_4(S_i,S_e)=0$  can be written as 
\begin{align}
\partial_{r} S_{e}(2) + \frac{\lambda}{4} S_{e}(2) 
&= \gamma \frac{1}{4} e^{\frac{\lambda}{2}} \int_{1}^{2} r^{2}h\left(\frac{2\lambda}{r^{2}}\right)e^{-\frac{\lambda }{2}(2-\frac{2}{r})}S_e(r) \,dr
\notag \\
&= a \gamma  \frac{\lambda}{2} e^{-\frac{\lambda}{2}} \int_{1}^{2}  e^{-\frac{b r^{2}}{2 \lambda}}e^{\frac{\lambda}{r}}S_e(r) \,dr.
\label{BC_Se}
\end{align}
Thus $B(\la,S_e)=0$, where $B$ was defined in  \eqref{eqUc}.

%It is easy to see that $\cS$ is a closed subset of $H^2_{r^2}(I)$.  

\subsection{Nullspace of $\sL$} \label{S3.1}

We show that the nullspace of $\sL$ is one-dimensional for the case $\la=\la^\dagger$, where $\la^\dagger>0$ is the sparking voltage.

\begin{lem}     \label{Null1}
Let $(a,b,\ga) \in \cS$   and $\la=\la^\dagger>0$.  
Then ${\rm dim} \, N(\sL)=1$.  
%For any $(a, b, \gamma) \in A$, there holds that ${\rm dim} \, N(\sL)=1$ for the case $\lambda=\lambda^{\dagger}$.
\end{lem}
					\begin{proof}
 Let $(a,b,\ga)\in\cS$ be fixed.    
 The linear operator $\cL$ depends on these parameters $a,b,\ga$ as well as $\la$.  
 We note that the second equation $\sL_{2} (S_e)=0$ 
together with $S_{e}(1)=0$ and \eqref{BC_Se} 
is identical to the boundary problem \eqref{eqUa}-\eqref{eqUb}.   
%except for the normalization $S_e(2)=1$.  
%So let us work with \eqref{eqU} and its solutions $u$.  
 Clearly the dimension of its nullspace is at most one.   
 Our goal is to prove that the nullspace is not trivial.                                                                                                                                                                                                                                                                                                                                                                                                                                                                                                                                                                                                                                                                                                                                                                                                                                                                                                                                                                                                                                     
%As pointed out previously, the solution of $\sL_{2} (S_e)=0$ with $S_{e}(1)=0$ 
%and $S_e(2)=1$ exists and is unique.   

Owing to the definition of $\la^\dagger$ as an infimum, there exists a sequence 
$\la_n \to \la^\dagger$ and  $u_n\in H^2$ satisfying \eqref{eqU} with $\la=\la_n$.  
We normalize $u_n$ by defining $v_n(r) = \|u_n\|^{-1}_{L^2}u_n(r)$.  
Then each $v_n$ also satisfies \eqref{eqU}.  
We claim that the $v_n$ are bounded in $H^2$.    
Indeed, this will follow by the energy method.  
Recall the inner weighted product defined in \eqref{inner-product1}.
We multiply \eqref{eqUa} by $r^2v_n$, 
{
integrate and use  $v_n(1)=0$ and \eqref{eqUc} to obtain 
\begin{align*}
\|v_n'\|^2_{L^2} & \le C (|v_n(2)|+\|v_n\|_{L^{2}}) |v_n(2)| + C\|v_n\|^2_{L^2}
\\
&\le C (\|v_n\|_{L^2}^{1/2}\|v_n'\|_{L^2}^{1/2} + \|v_n\|_{L^2}) \|u_n\|_{L^2}^{1/2}\|v_n'\|_{L^2}^{1/2} + C\|v_n\|^2_{L^2}.
\end{align*}
Then we use $\|v_n\|_{L^2}=1$ and Schwarz's inequality.}
So $v_n$ is bounded in $H^1$ and therefore by \eqref{eqUa} also in $H^2$.  

Now we take a weakly convergent subsequence of $\{v_n\}$ in $H^2$  to some limit $v$.  
The limit $v$ satisfies the ODE, as well as $v(1)=0$ and { $\|v\|_{L^2}=1$}.  
In addition, $v\ge0$ and $B(\la^\dagger,v)=0$. 
In case $v(r)$ vanishes at some point $r_0\in(1,2)$, then we would have  $v(r_0)=v'(r_0)=0$. 
It would follow that $v\equiv 0$, which contradicts { $\|v\|_{L^2}=1$}.  
Thus we infer that $v(r)>0$ for all $r \in I$.  

If we define the positive function $S_e=v$, then $\sL_2(S_e)=0$.  
We easily solve $\sL_1(S_i,S_e)=0$ for $S_i$.  
So $\sL_4(S_i,S_e)=0$.  Finally we easily solve $\sL_3(S_i,S_e,W)=0$ for $W$.  
Thus $(S_i,S_e,W)$ belongs to the nullspace of $\sL$.   So $\dim N(\sL)=1$.
\end{proof}

\subsection{Adjoint and range of $\sL$} \label{S3.2}

This section is devoted to computing the adjoint operator $\sL^*$ and to using it 
to determine the codimension of the range $R(\sL)$ of $\sL$.  
Recall that the linearized operator $\sL$ as an unbounded operator on 
$L^2\times L^2\times L^2$ with its domain $D(\sL) =  X \subset  H^1\times H^2\times H^2$ 
and its range $R(\sL)\subset L^2\times L^2\times L^2\times \real$.     
Thus its adjoint $\sL^*$ has its domain $D(\sL^*) \subset  L^2\times L^2\times L^2\times \real$ and its range $R(\sL^*) \subset L^2\times L^2\times L^2$.  
%The weighted $L^2$ inner product $\langle \cdot , \cdot \rangle$ is defined in \eqref{inner-product1}.

\begin{lem}
The adjoint operator $\sL^*$ is  given by 
%\begin{subequations}\lb{AdjointOp}
\begin{align*}
D(\sL ^*):=&\{(\psi_i,\psi_e,\psi_v,\psi_b)\in H^1 (I)\times H^2 (I) \times H^2 (I) \times \mathbb R \, ;  \, \er{domain2} \text{ holds}  \},
%\lb{domain1}
\\
\sL^*_1(\psi_i,\psi_v):=&
-2\frac{k_i}{r^{2}} \lambda \partial_r \psi_i-\psi_v,
%\lb{AdjointOp1}
\\
\sL^*_2(\psi_i,\psi_e,\psi_v):=&
-\frac{1}{r^{2}} \partial_{r} ( r^{2} \partial_{r} \psi_e )-g(\lambda \partial_{r} H)\psi_e
-k_eh(\lambda\partial_{r} H)e^{-\frac{\lambda}{2}H} \psi_i+e^{-\frac{\lambda}{2}H}\psi_v,
%\lb{AdjointOp2}
\\
\sL ^*_3(\psi_v):=&
\frac{1}{r^{2}} \partial_{r} ( r^{2} \partial_{r} \psi_v ),
%\lb{AdjointOp3}
\end{align*}
where the boundary condition for $\sL^*$ is 
\begin{equation}\lb{domain2}
4\psi_e(2)-\psi_b=\psi_e(1)=\psi_v(1)=\psi_v(2)
=4\psi_i(2) -\gamma\frac{1}{k_e} e^{\frac{\la}{2}}\psi_b
=  4\partial_r\psi_e(2)
+\frac{\lambda}{4}\psi_b=0.
\end{equation}
%\end{subequations}  
Hence $\pa_r\psi_e(2) + \tfrac\la 4 \psi_e(2) = 0$.
\end{lem}

\begin{proof}
The first thing to check is that  
\begin{align*}
\left\langle \sL _1(S_i,S_e),\psi_i\right\rangle + \left\langle \sL _2(S_e),\psi_e\right\rangle + &
\left\langle \sL _3(S_i,S_e,W),\psi_v\right\rangle + \sL _4(S_i,S_e) \psi_b
\\
&=\left\langle S_i,\sL_1^*\Psi \right\rangle + 
\left\langle S_e,\sL_2^*\Psi \right\rangle + 
\left\langle W,\sL_3^*\Psi \right\rangle 
\end{align*}
for all $\mathcal U:=(S_i,S_e,W) \in X=D(\sL)$ 
and $\Psi := (\psi_i,\psi_e,\psi_v,\psi_b) \in D(\sL^*)$,
where the weighted inner product $\langle \cdot , \cdot \rangle$ is defined in \eqref{inner-product1}.
We observe that
\begin{align}
& \left\langle \sL _1(S_i,S_e),\psi_i\right\rangle + \left\langle \sL _2(S_e),\psi_e\right\rangle + 
\left\langle \sL _3(S_i,S_e,W),\psi_v\right\rangle + \sL _4(S_i,S_e) \psi_b
\notag\\
&=-\left\langle S_i, 2 k_i\lambda r^{-2} \partial_x\psi_i \right\rangle 
+2S_i(2) k_i\lambda \psi_i(2)
-\left\langle S_e, k_eh(\lambda \partial_{r} H)e^{-\frac{\lambda}{2}H} \psi_i \right\rangle
\notag\\
&\quad -\left\langle S_e,  r^{-2} \partial_{r} ( r^{2} \partial_{r} \psi_e) \right\rangle
-4\partial_r S_e(2)\psi_e(2)
+\partial_r S_e(1)\psi_e(1)
+S_e(2)\partial_r\psi_e(2)
-\left\langle S_e, g(\lambda \partial_{r} H)\psi_e \right\rangle 
\notag\\
&\quad + \left\langle W, r^{-2} \partial_{r} ( r^{2} \partial_{r} \psi_v) \right\rangle
+4\partial_r W(2)\psi_v(2)-\partial_r W(1)\psi_v(1)
-\left\langle S_i, \psi_v \right\rangle 
+\left\langle S_e, e^{-\frac{\lambda}{2}H}\psi_v \right\rangle
\notag\\ 
&\quad +\partial_xS_e(2)\psi_b+S_e(2)\frac{\lambda}{4}\psi_b
-S_i(2)\gamma\frac{k_i}{k_e}\frac{\lambda}{2} e^{\frac{\la}{2}}\psi_b,   
\label{ELL1}
\end{align}
due to integration by parts and  \eqref{BC0}.
Grouping them with respect to $S_i$, $S_i(2)$, $S_e$, $\partial_rS_e(2)$,
$\partial_rS_e(1)$, $S_e(2)$, $W$, $\partial_rW(2)$, and $\partial_rW(1)$, 
and also using the boundary conditions \er{domain2}, we have
\begin{align*}
&\left\langle \sL _1(S_i,S_e),\psi_i\right\rangle + \left\langle \sL _2(S_e),\psi_e\right\rangle + 
\left\langle \sL _3(S_i,S_e,W),\psi_v\right\rangle + \sL _4(S_i,S_e) \psi_b
\\
&=-\left\langle S_i, 2k_i \lambda r^{-2} \partial_x\psi_i+\psi_v \right\rangle
\\
&
\quad -\left\langle S_e, k_eh(\lambda \partial_{r} H)e^{-\frac{\lambda}{2}H} \psi_i 
+r^{-2} \partial_{r} ( r^{2} \partial_{r} \psi_e) +g(\lambda \partial_{r} H)\psi_e- e^{-\frac{\lambda}{2}H}\psi_v\right\rangle
+\left\langle W, r^{-2} \partial_{r} ( r^{2} \partial_{r} \psi_v) \right\rangle
\\
&=\left\langle S_i,\sL_1^*\Psi \right\rangle + 
\left\langle S_e,\sL_2^*\Psi \right\rangle + 
\left\langle W,\sL_3^*\Psi \right\rangle.
\end{align*}

The second thing to check is that 
if $\langle \sL \mathcal U,\Psi \rangle = \langle \mathcal U,\zeta \rangle$  for all $\mathcal U\in D(\sL)$, 
then $\Psi\in D(\sL^*)$  and  $\sL^* \Psi = \zeta$.  
Assume now that $\langle \sL \mathcal U,\Psi \rangle = \langle \mathcal U,\zeta \rangle$  for all $\mathcal U\in D(\sL)$, where $\Psi\in L^2\times L^2\times L^2 \times \mathbb R$ and $\zeta:=(\zeta_{i},\zeta_{e},\zeta_{v})\in L^2\times L^2\times L^2$ are given. 
We begin with showing $\Psi \in H^{1} \times H^{2} \times  H^{2}  \times  \mathbb R$.
Recall the weighted inner product defined in \eqref{inner-product1}.
Taking  $\mathcal U=(0,0,W) \in C_{0}^{\infty} \times C_{0}^{\infty} \times C_{0}^{\infty}$, we see that $\sL^*_3 \Psi=\zeta_{v}$ holds in the distribution sense, i.e.
%\begin{gather*}
$\partial_{r} ( r^{2} \partial_{r} \psi_v )= r^{2} \zeta_{v} \in L^{2}$.
%\end{gather*}
This implies that $r^{2} \partial_{r} \psi_v - \int_{1}^{r}  t^{2}\zeta_{v}(t) \,dt$ is a constant. 
Therefore $\partial_{r} \psi_v$ belongs to $L^{2}$ and so do $\partial_{r}^{2} \psi_v$.
Thus $\psi_v \in H^{2}$.
In the same way, by choosing  $\mathcal U=(S_{i},0,0) \in C_{0}^{\infty} \times C_{0}^{\infty} \times C_{0}^{\infty}$, it is seen that 
%\begin{gather}
$-2 k_i \lambda \partial_r \psi_{i}  = r^{2} (\zeta_{i}+\psi_{v}) \in L^{2}$.
%\end{gather}
This means $\psi_i \in H^{1}$. Similarly, we can obtain  $\psi_e \in H^{2}$.

Within 
the set $D(\cL)=X$ we notice that each of the six quantities 
\[  
-S_e'(2),\  S'_e(1),\   -W'(1),\   \tfrac14W'(2),\   \tfrac{\la k_i}{2} S_i(2),\   S_e(2)    
\] 
 is completely {\it arbitrary}.  
Furthermore, looking at the expression for $\langle \sL \mathcal U,\Psi \rangle$ given in \eqref{ELL1}, 
we notice that the coefficients of these six quantities are, precisely and respectively, 
\begin{gather*}%\label{boundary-psi}
-4\psi_e(2) + \psi_b, \  \  \psi_e(1), \  \  \psi_v(1), \ \   \psi_v(2), \ \   
 \psi_i(2) - \frac{\ga }{4k_e} e^{\frac\la 2}\psi_b, \  \    
  \ps'_e(2)+\frac\la 4 \psi_b.
\end{gather*}
These boundary values have meaning due to $\Psi \in H^{1} \times H^{2} \times  H^{2}  \times  \mathbb R$.
That the six expressions vanish is precisely the same as the boundary conditions \eqref{domain2}.  
It follows that $\Psi= (\psi_i,\psi_e,\psi_v,\psi_b)\in D(\sL^*)$  and  $\zeta=\sL^* \Psi $. 
\end{proof}

\begin{lem}\label{codim1}
Let $(a, b, \gamma) \in \sS$ and $\la=\la^\dagger$.  
Then  ${\rm codim} \, R(\sL)=1$.  % for the case $\lambda=\lambda^{\dagger}$.  
\end{lem}

\begin{proof} 

We begin by computing  the nullspace $N(\sL^*)$ of the adjoint operator.  
%We reduce the boundary value problem associated by $\sL^*$.
From $\sL ^{*}_3=0$ and boundary conditions $\psi_v(0)=\psi_v(L)=0$,
we see that 
\begin{equation}\lb{psi_v}
\psi_v=0.
\end{equation}
From this and $\sL^*_1=\sL^*_2=0$, we have the equations
\begin{subequations}\lb{AdjointProblem}
\begin{gather}
\partial_r \psi_i=0,
\lb{AdjointEq1}\\
-\frac{1}{r^{2}} \partial_{r} ( r^{2} \partial_{r} \psi_e )-g(\lambda \partial_{r} H)\psi_e
-k_eh(\lambda\partial_{r} H)e^{-\frac{\lambda}{2}H} \psi_i=0.
\lb{AdjointEq2}
\end{gather}
Owing to \er{domain2} and substituting $\psi_b=4\psi_e(2)$, 
the boundary conditions for this system are  
\begin{gather}
\psi_i(2)-\frac{\gamma}{k_e} e^{\frac{\la}{2}}\psi_e(2)=0,
\lb{AdjointBC1}
\\
\psi_e(1)=0,
\lb{AdjointBC2}
\\
\partial_r \psi_e(2)+\frac{\lambda}{4}\psi_e(2)=0.
\lb{AdjointBC3}
\end{gather}
\end{subequations}
It remains to solve the problem \er{AdjointProblem}.  
We reduce it to  a scalar equation for $\psi_e$ alone, as follows.  
Using \er{AdjointEq1}  and \er{AdjointBC1}, we obtain
\begin{equation}\lb{AdjointEq3}
\psi_i(x) \equiv \frac{\gamma}{k_e} e^{\frac{\la}{2}}\psi_e(2).
\end{equation}
Plugging 
this into \er{AdjointEq2}, we have the problem for $\psi_e$:
\begin{equation}
-\frac{1}{r^{2}} \partial_{r} ( r^{2} \partial_{r} \psi_e )-g(\lambda \partial_{r} H)\psi_e
-k_eh(\lambda\partial_{r} H)e^{-\frac{\lambda}{2}H} \frac{\gamma}{k_e} e^{\frac{\la}{2}}\psi_e(2)=0
\lb{AdjointEq4} 
\end{equation}
together with \er{AdjointBC2} and \er{AdjointBC3}.

Thus the nullspace of $\sL^*$ is given by 
\begin{align*} 
N(\sL^*) = \{\Psi\in D(\sL^*)\ |\ \  &\psi_b = 4\psi_e(2),\ \psi_v\equiv 0,\ \ 
\psi_i\equiv \tfrac{\gamma}{k_e} e^{\frac{\la}{2}}\psi_e(2) , \\
&\psi_e \text{ solves } \er{AdjointEq4}, \er{AdjointBC2}, \er{AdjointBC3} \ \}.
\end{align*}
In particular, $N(\sL^*)$ is completely determined by the solutions of 
\begin{subequations}              \label{PhiEq1}
\begin{gather}
-\tfrac{1}{r^{2}} \partial_{r} ( r^{2} \partial_{r} \psi_e )-g(\lambda \partial_{r} H)\psi_e
-h(\lambda\partial_{r} H)e^{-\frac{\lambda}{2}H} {\gamma} e^{\frac{\la}{2}}\psi_e(2)=0,
\label{PhiEq2}\\
\psi_e(1)=\partial_{r} \psi_e(2) + \tfrac{\lambda}{4} \psi_e(2)=0, 
%- \gamma \frac{\lambda}{2} e^{-\frac{\lambda}{2}} \int_{1}^{2} a e^{-\frac{b r^{2}}{2 \lambda}}e^{\frac{\lambda}{r}}S_e(r) \,dr=0, 
\end{gather}
\end{subequations}
because $\psi_i, \psi_v, \psi_b$ are uniquely expressed in terms of $\psi_e$. 
In fact, if $\psi_e(2)=0$, then by \er{PhiEq1} the function $\psi_e$ satisfies a second-order equation with $\pa_r\psi_e(2) = \psi_e(2) = 0$, so that $\psi_e\equiv 0$.  
Thus if $\psi_e$ is not identically zero, we can normalize $\psi_e(2)=1$.

Next, because ${\rm codim} \, R(\sL)={\rm dim} \, N(\cL^{*})$, it suffices to consider 
the boundary value problem \eqref{PhiEq1} which defines $N(\cL^*)$.   
%Indeed, $\psi_i$ is given by \eqref{AdjointEq3} and $\psi_e(2)$ as well as $\psi_{v} \equiv 0$ holds. 
We denote by $\sM$ the operator associated to the boundary value problem \eqref{eqU}, which is the same as the boundary problem $\sL_{2} (S_e)=0$ with $S_{e}(1)=0$ and \eqref{BC_Se}.
Then the adjoint operator $\sM^{*}$ is the operator associated to \eqref{PhiEq1}.

Since the spectrum of $\sM$ is bounded below,  $\sM+\mu I$ is invertible for $\mu \gg 1$.   Moreover, its inverse $K=(\sM+\mu I)^{-1}$ is compact on $L^{2}(I)$.
Thus $I-\mu K = (\sM+\mu I)^{-1}\sM$ and $N(I-\mu K)=N(\sM)$.  
It is clear that $K^{*}=(\sM^{*}+\mu I)^{-1}$. 
Similarly, $N(I-\mu K^*)=N(\sM^*)$.  
Since $K$ is compact, we know from \cite{Ev} that 
$\dim N(I-\mu K)={\rm codim} N(I-\mu K^{*})$.  
Thus we have 
\[  
1= \dim N(\sL) = \dim N(\sM) =\dim N(\sM^*) = {\rm codim} R(\sM) =  {\rm codim} R(\sL).
\]
%On the other hand, $(I-\mu K)u=0$ is equivalent to $\tilde{\cL}_{2} u =0$.
%These with Lemma \ref{Null1} means that ${\rm dim} \, N(I-\mu K^{*})=1$ and so ${\rm dim} \, N(\tilde{\cL}^{*}_{2})=1$ for the case $\lambda=\lambda^{\dagger}$.
%Thus ${\rm codim} \, R(\sL)=1$.
\end{proof}

%%%%%%%%%%%%%%%%%%%%%%%%%%%%%%%%%%%
\subsection{Transversality condition}\label{S5}
In this section, we explicitly rewrite the transversality condition

\begin{equation}\label{transversality1}
\partial_\la \partial_{(\rho_i,R_e,V)}  \sF(\lambda^\dagger,0,0,0)[1,\varphi_i^\dagger,\varphi_e^\dagger,\varphi_v^\dagger]
\notin R(\partial_{(\rho_i,R_e,V)}\sF(\lambda^\dagger,0,0,0)),
\end{equation}
where $(\varphi_i^\dagger,\varphi_e^\dagger,\varphi_v^\dagger) \in N(\cL)$ and $\lambda^\dagger$ is the sparking voltage.

\begin{pro}
Let $(a,b,\ga) \in \cS$   and $\la=\la^\dagger>0$.  
The transversality condition \eqref{transversality1} is equivalent to 
the condition $F(a, b, \gamma) \neq 0 $ with $(\lambda,\varphi_{e})=(\lambda^{\dagger},\varphi_{e}^{\dagger})$.
\end{pro}

\begin{proof}
Recall from \eqref{transversality0} that
\begin{align}    
F(a,b,\gamma)&=-{\gamma} e^{\frac{\la^{\dagger}}{2}}\psi_e(2)
\int_1^2 r^{2} \left\{ h'\left( \lambda^{\dagger} \partial_{r} H \right) (\partial_{r} H) - h\left( \lambda^{\dagger} \partial_{r} H \right)\frac{H}{2} \right\}e^{-\frac{\lambda^{\dagger}}{2}H}\varphi_e^\dagger \,dr 
   \notag\\
&\quad -\int_1^2 r^{2} \psi_{e}g'(\lambda^{\dagger} \partial_{r} H) (\partial_{r} H) \varphi_e^\dagger\, dr+\psi_e(2)          
\left\{\varphi_e^\dagger(2) - 2\partial_r\varphi_e^\dagger(2)  - \frac{\lambda^\dagger}  {2}\varphi_e^\dagger(2)\right\}.
\label{transversality2}
\end{align}
By differentiating \er{LinearEq0} with respect to $\lambda$, we see that
\begin{align*}
f_i&:=\partial_\la \partial_{(\rho_i,R_e,V)} \sF_1(\lambda,0,0,0)[1,\varphi_i^\dagger,\varphi_e^\dagger,\varphi_v^\dagger]              \notag\\
&\ =2\frac{k_i}{r^{2}} \partial_{r} \varphi_i^\dagger
-k_e \left\{ h'\left( \lambda^{\dagger} \partial_{r} H \right) (\partial_{r} H) - h\left( \lambda^{\dagger} \partial_{r} H \right)\frac{H}{2} \right\}e^{-\frac{\lambda^{\dagger}}{2}H}\varphi_e^\dagger,
%\lb{LinearEq11}
\\
f_e&:=\partial_\la \partial_{(\rho_i,R_e,V)} \sF_2(\lambda,0,0,0)[1,\varphi_i^\dagger,\varphi_e^\dagger,\varphi_v^\dagger]=
- g'(\lambda^{\dagger} \partial_{r} H) (\partial_{r} H) \varphi_e^\dagger,
%\lb{LinearEq12}
\\
f_v&:=\partial_\la \partial_{(\rho_i,R_e,V)} \sF_3(\lambda,0,0,0)[1,\varphi_i^\dagger,\varphi_e^\dagger,\varphi_v^\dagger] =
-e^{-\frac{\lambda^{\dagger}}{2}H}\frac{H}{2}\varphi_e^\dagger,
%\lb{LinearEq13}
\\
f_b&:=\partial_\la \partial_{(\rho_i,R_e,V)} \sF_4(\lambda,0,0,0)[1,\varphi_i^\dagger,\varphi_e^\dagger,\varphi_v^\dagger]=
\frac{1}{4} \varphi_e^\dagger(2) - \gamma \frac{k_{i}}{k_{e}} e^{\frac{\lambda^{\dagger}}{2}} \left(  \frac{\lambda^{\dagger}}{4} +  \frac{1}{2} \right) \varphi_i^\dagger(2).
%\lb{LinearEq14}
\end{align*}
On the other hand, the range $R(\partial_{(\rho_i,R_e,V)}\sF(\lambda,0,0,0))$ can be represented as 
\begin{gather}
R(\sL)=\left\{(f_i,f_e,f_v,f_b) \in Y \, ;  \, \er{Orthogonal1} \right\}, 
%\lb{reperesent1}
\notag\\
\int_{1}^{2} r^{2} (f_i\psi_i + f_e\psi_e + f_v\psi_v) dr + f_b\psi_b=0 \quad
\text{for  $(\psi_i,\psi_e,\psi_v,\psi_b)\in N(\sL^*)$}. 
\lb{Orthogonal1}
\end{gather}
		Owing 
to these formulas together with \er{domain2}, \er{psi_v}, and \er{AdjointEq3},
the transversality condition \er{transversality1} can be written as 
$(f_i,f_e,f_v,f_b) \not\in R(\cL)$.  That is, 
\[ 
 \psi_i\int_1^2 r^2f_i\ dr  + \int_1^2 r^2 \psi_e\ f_e\ dr  +  \psi_b\ f_b  \ne0.  \]
Explicitly, this means 
\begin{multline}
\frac{\gamma}{k_e} e^{\frac{\la^{\dagger}}{2}}\psi_e(2)
\int_1^2 r^{2} \left[ 
2\frac{k_i}{r^{2}} \partial_{r} \varphi_i^\dagger
-k_e \left\{ h'\left( \lambda^{\dagger} \partial_{r} H \right) (\partial_{r} H) - h\left( \lambda^{\dagger} \partial_{r} H \right)\frac{H}{2} \right\}e^{-\frac{\lambda^{\dagger}}{2}H}\varphi_e^\dagger
\right] \,dr 
\\
-\int_1^2 r^{2} \psi_{e}g'(\lambda^{\dagger} \partial_{r} H) (\partial_{r} H) \varphi_e^\dagger\, dr
+4\psi_e(2)\left\{\frac{1}{4} \varphi_e^\dagger(2) - \gamma \frac{k_{i}}{k_{e}} e^{\frac{\lambda^{\dagger}}{2}} \left(  \frac{\lambda^{\dagger}}{4} +  \frac{1}{2} \right) \varphi_i^\dagger(2)\right\}
\neq 0.
\lb{transversality3}
\end{multline}
				This is what we have to prove.  
However, among the six terms in \er{transversality3}, the first and last terms add up to  
\begin{align*}
& 2\gamma \frac{k_{i}}{k_e} e^{\frac{\la^{\dagger}}{2}}\psi_e(2) \int_1^2 \partial_{r} \varphi_i^\dagger \,dr
-\gamma \frac{k_{i}}{k_{e}} e^{\frac{\lambda^{\dagger}}{2}} ( \lambda^{\dagger} + 2 ) \psi_e(2)\varphi_i^\dagger(2) 
\\
&=-\gamma \frac{k_{i}}{k_{e}} e^{\frac{\lambda^{\dagger}}{2}}\lambda^{\dagger}\psi_e(2)\varphi_i^\dagger(2)
=-\psi_e(2)\left\{2\partial_r\varphi_e^\dagger(2)
    +  \frac{\lambda^\dagger}{2}  \varphi_e^\dagger(2) \right\}. 
\end{align*}
The last equality is due to  \er{LinearEq4} and the fact that 
$(\varphi_i^{\dagger},\varphi_e^{\dagger},\varphi_v^{\dagger}) \in N(\partial_{(\rho_i,R_e,V)}\sF(\lambda^\dagger,0,0,0))$.  
Substituting this simple equality into \er{transversality3} 
shows that the transversality condition \er{transversality2} is precisely the same as 
$F(a,b,\gamma)\neq 0$.
The proof is complete.
\end{proof}

%%%%%%%%%%%%%%%%%%%%%%%%%%%%%%%%%%%

\section{Global Bifurcation}
In this section, we prove the global bifurcation of the stationary solutions in Theorem \ref{mainthm0} by invoking the following abstract theorem \cite{CS1}.

\begin{thm}[\!\cite{CS1}]\lb{Global0}
Let $X$ and $Y$ be Banach spaces,
$\sO$ be an open subset of ${\mathbb R}\times X$ and $\sF:\sO \to Y$ be a real-analytic function. 
Suppose that
\begin{enumerate}[{(H}1{)}]
\item $(\lambda,0)\in \sO$ and $\sF(\lambda,0)=0$ for all $\lambda \in \mathbb R$;
\item for some $\lambda^\dagger \in \mathbb R$, $N(\partial_u\sF(\lambda^\dagger,0))$ and 
$Y\backslash R(\partial_u\sF(\lambda^\dagger,0))$ are one-dimensional, 
with the null space generated by $u^\dagger$, which satisfies  the transversality condition
\[
\partial_\la\partial_u\sF(\lambda^\dagger,0)(1,u^\dagger)\notin R(\partial_u\sF(\lambda^\dagger,0)),  
\]
where $\partial_u$ and $\partial_\la\partial_u$ mean Fr\'echet derivatives for $(\la,u) \in \sO$,
and $N(\sL)$ and $R(\sL)$ denote the null space and range of a linear operator $\sL$ between
two Banach spaces;
\item $\partial_u\sF(\lambda,u)$ is a Fredholm operator of index zero for any $(\la,u) \in \sO$
that satisfies the equation $\sF(\lambda,u)=0$;
\item for some sequence $\{\sO_j\}_{j\in \mathbb N}$ of bounded closed subsets of $\sO$ with
$\sO=\cup_{j\in \mathbb N} \sO_j$, 
the set $\{(\la,u) \in \sO \, ; \, \sF(\lambda,u)=0\}\cap \sO_j$ is compact for each $j\in\mathbb N$.
\end{enumerate}

Then there exists in $\sO$ a continuous curve 
${\sK}=\{(\la(s),u(s)) \, ; \, s \in \mathbb R\}$
of $\sF(\lambda,u)=0$ such that:
\begin{enumerate}[{(C}1{)}]
\item $(\la(0),u(0))=(\la^\dagger,0)$;
\item $u(s)=su^\dagger+o(s)$ in $X$  as $s \to 0$;
\item there exists a neighborhood $\sW$ of $(\la^\dagger,0)$
and $\ve > 0$ sufficiently small such that
\begin{equation*}
\{(\la,u)\in \sW \, ; \, u\neq 0 \text{ and } \sF(\la,u)=0 \}
=\{(\la(s),u(s)) \, ; \, 0<|s|<\ve\};
\end{equation*}
\item $\sK$ has a real-analytic reparametrization locally around each of its points;
\item one of the following two alternatives occurs:
\begin{enumerate}[(I)]
\item for every $j\in\mathbb N$, there exists $s_j>0$ such that $(\la(s),u(s))\notin \sO_j$
for all $s \in \mathbb R$ with $|s|>s_j$;
\item there exists $T>0$ such that
$(\la(s),u(s))=(\la(s+T),u(s+T))$ for all $s \in \mathbb R$.
\end{enumerate}
\end{enumerate}
Moreover, such a curve of solutions of $\sF(\lambda,u)=0$
having the properties (C1)-(C5) is unique (up to reparametrization).
\end{thm}
In general, there can be multiple secondary bifurcations but, due to its analyticity stated in (C5) above, the curve $\cK$ is unique.  

In order to apply the preceding theorem within the context of our problem, we define the sets
\begin{subequations}\label{Def-sO}
\begin{align}
\sO:=&\{(\lambda,\rho_i,R_e,V)\in (0,\infty)\times X \, ; \, 
 |\partial_r V+\lambda \partial_{r} H|>0   \} \ 
=\ \bigcup_{j\in \bbn} \sO_j ,\text{ where } 
\\
\sO_j:=&\left\{(\lambda,\rho_i,R_e,V)\in (0,\infty)\times X \, ; \, \ 
\lambda+\|(\rho_i,R_e,V)\|_{X}\leq j ,  \  
\lambda\geq \tfrac1j,  \ 
|\partial_r V+\lambda \partial_{r} H|   \geq \tfrac1j\right\}.
\end{align}  
\end{subequations}
					Note that 
$\sO$ is an open set and each $\sO_j$ is a closed bounded subset of $\sO$.
Furthermore, the $\sF_j$ are real-analytic operators because they are 
polynomials in $(\lambda,\rho_i, R_e, V)$ and their $r$-derivatives, 
except for the factor $h(|\partial_r V+\lambda \partial_{r} H|)$.  However, $\partial_r V+\lambda \partial_{r} H \neq 0$ 
in $\sO$ and the function $s\to h(|s|)$ is analytic for $s \neq 0$.

Hypothesis $(H1)$ is trivial.
In Theorem \ref{mainthm0}, we assume the hypothesis $(H2)$.
The other hypotheses $(H3)$ and $(H4)$ are validated in Lemmas \ref{index1} and \ref{cpt1} below,
respectively.
For that purpose, consider the linearized operator around an arbitrary triple of functions 
$(\rho_i^0,R_e^0,V^0)\in X$.  
\begin{lem}\lb{index1} 
For any $(\lambda,\rho_i^0,R_e^0,V^0)\in {\sO}$, 
the Fr\'echet derivative 
$\sL^0 = \partial_{(\rho_i,R_e,V)} \sF(\lambda,\rho_i^0,R_e^0,V^0)$
is a linear Fredholm operator of index zero from $X$ to $Y$.%, provided that $D(\lambda)=0$. 
\end{lem}
\begin{proof}
Recalling \eqref{Def-sO}  we  see that
$\partial_{r} V^{0}+\lambda \partial_{r} H > 0$
for any fixed choice of $(\lambda,\rho_i^0,R_e^0,V^0)$.
The operator $\sL^0=(\sL_1,\sL_2,\sL_3,\sL_4)$ acting linearly on the triple $(S_i,S_e,W)\in X$  has the form 
\begin{align}
\sL_1 = \sL_1(S_i,S_e,W) =& 
k_{i} (\partial_r V^0 + \lambda \partial_r H)\partial_{r} S_{i} + b_1S_i + b_2\partial_r^2 W + b_3S_e + c_{1}\partial_r W , 
\lb{sL1}\\
\sL_2 = \sL_2(S_i,S_e,W) =& 
- \partial_r^{2} S_e  + a_1\partial_r S_e + b_4 S_e + b_5 \partial_r^2W + b_6 \partial_rW  ,   
\lb{sL2}\\
\sL_3 = \sL_3(S_i,S_e,W) =&
\frac{1}{r^{2}}\partial_{r} (r^{2}\partial_r W)  + a_2 S_i + a_3 S_e ,
\lb{sL3} \\
\sL_4 = \sL_4(S_i,S_e,W) =&
\pa_r S_e(2) + (\tfrac\lambda2 \pa_r H(2) +\pa_rV^0(2))S_e(2)  +  \pa_rW(2) R_e^0(2) 
\notag \\
&-\gamma \frac{ k_i}{k_e} e^{\tfrac{\lambda}{2}H(2)} \left[ \{\pa_rV^0(2) + \lambda \pa_r H(2) \}S_i(2)  
+  \pa_rW(2)\rho_i^0(2) \right] , 
\lb{sL4}
\end{align}
where the coefficients $a_1,a_2,a_3$ belong to $C^1(\bar{I})$,
the coefficients $b_1,...,b_6$ belong to $C^0(\bar{I})$,
and the coefficient $c_{1}$ belongs to $L^{2}(I)$.

We first claim  that the linear operator $\sL^0$ satisfies the estimate 
\bqn  \label{ellest} 
\|(S_i,S_e,W)\|_X  \le  C\|\sL^0(S_i,S_e,W)\|_Y  +  C\|(S_i,S_e,W)\|_Y
\eqn
for all $(S_i,S_e,W)\in X$ and 
for some constant $C$ depending only on $(\lambda,\rho_i^0,R_e^0,V^0)$.
%where 
%\[
%Z:= L^{2}(I) \times L^{2}(I) \times H^{1}(I).
%\]
Indeed, applying a standard energy method  to \er{sL3} with $W(1)=W(2)=0$ leads to 
 \begin{align}
 \|\pa_r W\|_{H^{1}} 
 & \leq  C \| a_3S_i + a_4S_e-\sL_3\|_{L^{2}}
 \leq  C\|\sL^0(S_i,S_e,W)\|_Y  +  C\|(S_i,S_e,W)\|_Y.
 \lb{estsL3}
\end{align}
Keeping in mind that  $\partial_{r} V^{0}+\lambda \partial_{r} H \ge 1/j$, we see from \er{sL1} and \er{estsL3} that  
$S_i$ can be estimated by 
 \begin{align*}
 \|\pa_r S_i\|_{L^{2}} 
 &=\|  \{k_{i} \partial_r (V^0 + \lambda H)\}^{-1}
 \left(b_{1} S_i
 +b_2 \pa_r^2 W + b_3S_e + c_{1}\pa_r W-\sL_1\right)\|_{L^{2}}
 \notag\\
 &\leq C(\|S_i\|_{L^{2}}+\|S_e\|_{L^{2}}+\|\partial_{r} W\|_{H^1}+\|\sL_1\|_{L^{2}})
 \notag\\
 &\leq  C\|\sL^0(S_i,S_e,W)\|_Y  +  C\|(S_i,S_e,W)\|_Y.
% \lb{estsL1}
 \end{align*}
  			By writing \er{sL2} as 
  \begin{gather}\label{sL2*}
\partial_r^2 S_e  - \partial_r (a_1S_e)=  -\sL_2^0  -(\partial_r a_1)S_e + b_4S_e + b_5 \partial_r^2W 
+a_6\partial_rW,    
  \end{gather}
			we see that the left side is bounded in $L^2$ by  
\begin{gather}\lb{estsL2}
 \|\partial_r (\partial_r S_e - a_{1} S_{e})\|_{L^{2}}  \le  C\|\sL^0(S_i,S_e,W)\|_Y  +  C\|(S_i,S_e,W)\|_Y .  
\end{gather}
On the other hand, \eqref{sL4} with $|S_{e}(2)|^{2} \le C\|S_{e}\|_{L^2}  \|\partial_{r} S_{e}\|_{L^2}$ gives
\begin{gather*}
|\partial_{r} S_{e} (2)| \leq C(\|S_{e}\|_{L^2}^{1/2}  \|\partial_{r} S_{e}\|_{L^2}^{1/2} +  \|\sL^0(S_i,S_e,W)\|_Y  +  \|(S_i,S_e,W)\|_Y).
\end{gather*}
Integrating \eqref{sL2*} from $r$ to $2$, we find that the $L^{2}$-norm of $\pa_rS_e$ is bounded by the right hand side of \eqref{ellest}. Then  the $L^{2}$-norm of $\pa_r^{2}S_e$ is also bounded due to \eqref{estsL2}.
%Now the $L^{2}$-norm of $\pa_r^{3}W$ is bounded as well.
Furthermore, there holds that
  \begin{align*}
 \|\pa_r^3 W\|_{L^{2}} 
 =\|\partial_{r} (a_{2} \pa_r W +  a_3S_i + a_4S_e-\sL_3)\|_{L^{2}}
 \leq  C\|\sL^0(S_i,S_e,W)\|_Y  +  C\|(S_i,S_e,W)\|_Z.
% \lb{estsL3*}
 \end{align*}
The preceding estimates on $S_i$, $W$, and $S_e$ prove \er{ellest}.

The estimate \eqref{ellest} implies that the nullspace of $\sL^0$ has finite dimension and the range of $\sL^0$ is closed.  This is a general fact about linear operators; see for instance  \cite[Example 4 in $\S$3.12]{Zei1}. 
Therefore $\sL^0$ is a semi-Fredholm operator. 
Lemmas~\ref{Null1} and \ref{codim1} ensure that at the bifurcation point the 
 nullspace of $\partial_{(\rho_i,R_e,V)} \sF(\lambda^{*},0,0,0)$ has dimension one and 
the codimension of its range  is also one, 
so that its index is zero.    Since $\sO$ is connected 
and the index is a topological invariant \cite[Theorem 4.51, p166]{AA1}, 
$\sL^0$ also has index zero.  
This implies that the codimension of $R(\sL^0)$ is also finite.
This completes the proof of Lemma \ref{index1}. 
\end{proof}

%%%%%%%%%%%%%%  COMPACTNESS  %%%%%%%%%%%%%%%
\begin{lem}\lb{cpt1}
For each $j\in \mathbb N$, the set 
$K_j  =  \{(\lambda,\rho_i,R_e,V)\in {\sO_j} \, ; \, \ \sF(\lambda,\rho_i,R_e,V)=0\}$
is  compact in ${\mathbb R}\times X$.
\end{lem}
\begin{proof}
Let  $\{(\lambda_n,\rho_{in},R_{en},V_n)\}$ be any sequence in  $K_j$. 
It suffices to show that  it has a convergent subsequence
whose limit also belongs to $K_j$.  
By the assumed bound 
$|\lambda_n|+\|(\rho_{in},R_{en},V_n)\|_{X}\leq j$, 
there exists a subsequence, still denoted by $\{(\lambda_n,\rho_{in},R_{en},V_n)\}$,
and $(\lambda,\rho_i,R_e,V)$ such that
\begin{equation}\lb{converge1}
\lambda_{n}  \to  \lambda \  \text{in $\mathbb R$}, \quad 
\rho_{in}  \to  \rho_i \ \text{in $L^{2}(I)$}, \quad 
R_{en}  \to  R_{e} \ \text{in $H^{1}(I)$}, \quad
V_{n}  \to  V  \ \text{in $H^{2}(I)$}.
\end{equation}
Furthermore, 
\[
\partial_{r} V+\lambda \partial_{r} H\geq \tfrac1j.
\]
Since $\sO_j$ is closed in $X$,
it remains to show that 
\begin{gather*}
\sF_j(\lambda,\rho_i,R_e,V)=0 \qu \text{for $j=1,2,3,4,$}
\\
\rho_{in}\to \rho_i \ \text{in} \ H^1(I), \qu
R_{en}\to R_e \ \text{in} \ H^2(I), \qu
V_n\to V \ \text{in} \ H^3(I).
\end{gather*}

Now the first equation $\sF_1(\lambda_n,\rho_{in},R_{en},V_n)=0$ with $\rho_{in}(1)=0$ is equivalent to
\[
 \rho_{in}(r)=\frac{k_e}{k_i}\frac{1}{r^{2}(\partial_r V_n(r)+\lambda_n \partial_r H(r))}
\int_{1}^r t^{2} h(|\partial_r V_n(t)+\lambda_n \partial_r H(t)|)e^{-\frac{\lambda_n}{2}H(t)}R_{en}(t)\,dt.
\]
Taking the limit and using \er{converge1}, we see that 
\[
 \rho_{i}(r)=\frac{k_e}{k_i}\frac{1}{r^{2} (\partial_r V(r)+\lambda \partial_r H(r))}
\int_{1}^r t^{2} h(|\partial_r V(t)+\lambda\partial_r H(t)|)e^{-\frac{\lambda}{2}H(t)}R_{e}(t)\,dt,
\]
where the right hand side converges in $H^1(I)$.
Hence, we see that $\sF_1(\lambda,\rho_i,R_e,V)=0$ and $\rho_{in}\to \rho_i$ in $H^1(I)$.
Taking the limit using \eqref{converge1} in the third equation $\sF_3(\lambda_n,\rho_{in},R_{en},V_n)=0$ immediately leads to 
\[
 r^{-2} \partial_r (r^{2} \partial_r V)=\rho_i-e^{-\frac{\lambda}{2}H}R_{e}.
\]
Hence $\sF_3(\lambda,\rho_i,R_e,V)=0$ and $V_n\to V$ in $H^{3}(I)$.
The second equation $\sF_2(\lambda_n,\rho_{in},R_{en},V_n)=0$ 
can be written as 
\[   
\partial_r \{r^{2} (\partial_r R_{en} - R_{en} \partial_r V_n)\} 
=r^{2}\left\{\frac{\lambda_{n}}{2} \partial_{r} V_{n} \partial_{r} H
+\frac{\lambda_{n}^2}{4}|\partial_{r} H|^{2}
-h\left(|\partial_{r} V_{n}+\lambda \partial_{r} H|\right)\right\}R_{en}.
\]
Because the right side converges in $H^{1}(I)$, we see that 
$\pa_r R_{en} - R_{en} \pa_r V_n$ converges in $H^1(I)$. 
But $R_{en} \pa_r V_n$ converges in $H^1(I)$.  
Hence $\pa_rR_{en}$ converges in $H^1(I)$, 
which means that $R_{en}$ converges to some $R_e$ in $H^{2}(I)$.  
It is obvious from \er{converge1} and $\sF_4(\lambda_n,\rho_{in},R_{en},V_n)=0$
that $\sF_4(\lambda,\rho_i,R_e,V)=0$ is valid for the limit.
\end{proof}

%%%%%%%%%%%%%%%%%%%
		As we have checked all the conditions in Theorem \ref{Global0}, 
		the following global bifurcation conclusion is valid.
\begin{thm}       \lb{Global1}   
Let $(a,b,\ga)\in\cS$.  Define $\la^\dagger$ by \eqref{spv1}.   
Assume the transversality condition \er{TC1}. 
%Assume that the sparking voltage $\lambda^\dagger$, defined by \eqref{spv1}, exists.
Then there exists in the open set $\sO$ a continuous curve 
${\sK}=\{(\la(s),\rho_i(s),R_e(s),V(s)) \, ; \, s \in \mathbb R\}  \subset  \mathbb R\times X$
of stationary solutions to problem \er{sp0} such that
\begin{enumerate}[\rm {(C}1{)}]
\item $(\la(0),\rho_i(0),R_e(0),V(0))=(\lambda^\dagger,0,0,0)$;
\item $(\rho_i(s),R_e(s),V(s))=s(\varphi_i^\dagger,\varphi_e^\dagger,\varphi_v^\dagger)+o(s)$ in the space $X$
 as $s \to 0$, where $(\varphi_i^\dagger,\varphi_e^\dagger,\varphi_v^\dagger)$ is
a basis of $N(\partial_{(\rho_i,R_e,V)}\sF(\lambda^\dagger,0,0,0))$, and $\varphi_i^\dagger(r),\varphi_e^\dagger(r)>0$ holds for $r \in (1,2]$;
\item there exists a neighborhood $\sW$ of $(\lambda^\dagger,0,0,0)$
and $\ve < 1$ such that
\begin{multline*}
\{(\la,\rho_i,R_e,V)\in \sW \, ; \, (\rho_i,R_e,V)\neq (0,0,0), \ \sF(\la,\rho_i,R_e,V)=0 \}
\\
=\{(\la(s),\rho_i(s),R_e(s),V(s)) \, ; \, 0<|s|<\ve\};
\end{multline*}
\item $\sK$ has a real-analytic reparametrization locally around each of its points;
\item at least one of the following four alternatives occurs:
\begin{enumerate}[\rm (a)]
\item $\varliminf_{s \to \infty}\la(s)= 0$;
\item $\varliminf_{s \to \infty}\inf_{r \in I}\{\partial_r V(s,r)+\la(s) \partial_r H(r)\}
 \equiv   \varliminf_{s \to \infty}\inf_{r \in I}\{\partial_r \Phi(s,r)\} 
 =0$;
%\item $\varlimsup_{s \to \infty}\la(s)=\infty$;
\item $\varlimsup_{s \to \infty}(\|\rho_i\|_{H^{1}}+\|R_e\|_{H^{2}}+\|V\|_{H^{3}}+\lambda)(s)=\infty$;
\item there exists $T>0$ such that
$$(\la(s),\rho_i(s),R_e(s),V(s))=(\la(s+T),\rho_i(s+T),R_e(s+T),V(s+T))$$ for all $s \in \mathbb R$.
\end{enumerate}
\end{enumerate}
Moreover, such a curve of solutions to problem \er{sp0} having the properties {\rm (C1)--(C5)}
is unique (up to reparametrization).
\end{thm}

Conditions (C1)--(C3)  express the local bifurcation, while (C4)--(C5) are assertions about the 
global curve $\sK$.  Alternative (c)  asserts that $\sK$ may be unbounded.  
Alternative (d) asserts that $\sK$ may form a closed curve (a `loop').  

%%%%%%%%%%%%%%%%%%%%%%%%%%%%%%%%%%%
\section{Positive Densities}\label{SPD}

Note that Theorem \ref{Global1} does not assert any positivity of $R_e$ or $\rho_i$ along $\cK$. 
However, we should of course keep in mind that for the physical problem  $\rho_i$ and $R_e$ are densities 
of particles and so they should be non-negative. 
In this section we investigate the part of the mathematical curve $\sK$ that corresponds to such densities. 
We will often abbreviate  $\partial_r V(s,r)+\lambda(s) \partial_r  H(r)$ as $\partial_r \Phi (s,r)$, and
suppress the variable $r$, as in $\rho_i(s)=\rho_i(s,\cdot)$, $R_e(s)=R_e(s,\cdot)$, $V(s)=V(s,\cdot)$, and $\partial_r \Phi (s,\cdot)=\partial_r \Phi (s)$.

The content of this section is the following theorem, which states in particular that 

\noindent   {\it either (i)} 
$\rho_i$ and $R_e$ remain positive 

\noindent   {\it or (ii)} the curve of positive solutions 
forms a half-loop going from $\lambda^\dagger$ to some other voltage $\lambda^\ddagger$.

Here $\lambda^\dagger$ is defined in \er{spv1} 
and $\lambda^\ddagger$ is a voltage with $\lambda^\dagger<\lambda^\ddagger$.
We emphasize  that the curve $\sK$ is never the half-loop {\it (ii)}
unless a voltage $\lambda^\ddagger > \lambda^\dagger$ exists. 

\begin{thm}\lb{Global2}   
Let $(a,b,\ga)\in\cS$.  Define $\la^\dagger$ by \eqref{spv1}. Assume the transversality condition \er{TC1}. 
For the global bifurcation curve $\sK = (\la(s),\rho_i(s),R_e(s),V(s))$ in Theorem \ref{Global1},
one of the following two alternatives occurs:
\begin{enumerate}[\rm (i)] 
\item  
%$\rho_i(s,r)>0$ and $R_e(s,r)>0$ \  for all $0<s<\infty$ and $r\in (1,2]$.  
both $\rho_i(s,r)$ and $\rho_e(s,r)=(R_ee^{-\la H/2})(s,r)$ 
are positive  for all $s\in(0,\infty)$ and $r \in (1,2]$.   Moreover, 
\begin{gather}\label{glow1}
\varlimsup_{s\to\infty}\{\|\rho_i(s)\|_{C^0}+\|\rho_e(s)\|_{C^0}+\lambda(s)\}=\infty;
\end{gather}
\item 
there exists a voltage $\lambda^\ddagger$ satisfying  $\lambda^\dagger<\lambda^\ddagger$  
and a finite parameter value $s^\ddagger>0$ such that 
\begin{enumerate}[\rm (1)]
\item $\rho_i(s,r)>0$ and $R_e(s,r)>0$ \ for all $s\in(0,s^\ddagger)$ and $r \in (1,2]$;
\item $(\la(s^\ddagger),\rho_i(s^\ddagger),R_e(s^\ddagger),V(s^\ddagger))=(\lambda^\ddagger,0,0,0)$.
%\item $(\rho_i(s),R_e(s))=(s^\ddagger-s)(\varphi_i^\ddagger,\varphi_e^\ddagger)+o(|s-s^\ddagger|)$
%as $s\nearrow s^\ddagger$, where $(\varphi_i^\ddagger,\varphi_e^\ddagger)$ 
%is a basis of $N(\partial_{(\rho_i,R_e,V)}\sF(V_c^\ddagger/L,0,0,0))$, and $\varphi_i^\ddagger,\varphi_e^\ddagger>0$ holds on the interval $I$.
%\item  $\rho_i(s,x)<0$ and $R_e(s,x)<0$ \ for $0<s-s^\ddagger\ll1$ and $x \in (0,L]$.
\end{enumerate}
\end{enumerate}
\end{thm}

\begin{proof}
Let us define  
\begin{equation}   \label{Tddagger}
 s^\ddagger:=\inf\{s>0 \, ; \, R_e(s,r_0)=0 \ \text{for some $r_0 \in (1,2]$}\}.
\end{equation}    
Clearly $R_e>0$ in $(0,s^\ddagger) \times (1,2]$.     
By (C2) in Theorem \ref{Global1}, $s^\ddagger>0$. 
				In case $s^\ddagger=\infty$, 
then $R_e>0$ in $(0,\infty) \times (1,2]$.
Also  $\partial_r \Phi = \partial_r V+\lambda \partial_r  H \ne0$ due to $(\la(s),\rho_i(s),R_e(s),V(s)) \in \sO$.   
Since $\partial_r \Phi $ goes from 0 to $\la>0$,  it follows that $\partial_r \Phi  >0$ in $\bar{I}$.  
Thus the formula 
\begin{equation}\label{R_i1}
 \rho_{i}(r)=\frac{k_e}{k_i} \frac{1}{r^{2}\partial_r \Phi (r) }
\int_{1}^r t^{2} h(|\partial_r \Phi (r)|)e^{-\frac{\lambda}{2}H(t)}R_{e}(t)\,dt,
\end{equation}
 yields $\rho_i>0$.  
This proves the first line of alternative (i) in the theorem in case $s^\ddagger=\infty$.  
We will postpone the proof of \eqref{glow1} for a while.

Assuming now that $s^\ddagger<\infty$, we will show that (ii) happens.  
We will first prove that $R_e(s^\ddagger,\cdot)$ vanishes identically.  
By \er{Tddagger},   
$R_e(s^\ddagger,\cdot)$ takes the value zero, which is its minimum, 
at some point $r_0\in \bar{I}=[1,2]$. 
In  case $r_0$ belongs to the open interval $I=(1,2)$, 
$\partial_r R_e(s^\ddagger,r_0)=0$ also holds.
Solving $\sF_2(\lambda,\rho_i,R_e,V)=0$ with $R_e(s^\ddagger,r_0)=\partial_r R_e(s^\ddagger,r_0)=0$,
we see by uniqueness that $R_e(s^\ddagger)\equiv 0$.  
On the other hand, in case $r_0=1$, by \er{Tddagger} there exists a sequence $\{(s_n,r_n)\}_{n \in \mathbb N}$ such that 
$R_e(s_n,r_n)=0$ with $s_n \searrow s^\ddagger$ and $r_n \searrow 0$.
Rolle's theorem ensures that there also exists some $\tilde{r}_n\in(0,r_n)$ 
such that $\partial_r R_e(s_n,\tilde{r}_n)=0$.
Letting $n\to \infty$, we see that $\tilde{r}_n\to 0$ and thus $\partial_r R_e(s^\ddagger,0)=0$.
Hence we again deduce  by uniqueness that $R_e\equiv 0$.  
Thirdly, in case $r_0=2$, it is obvious that $\partial_r R_e(s^\ddagger,2)\leq 0$.
On the other hand, we see from $\sF_4=0$ and \er{R_i1} that
\[
\partial_rR_e(s^\ddagger,2)=\gamma\frac{k_i}{k_e}e^{\frac{\la}{2}H}(\partial_r \Phi)\ro_i(s^\ddagger,2)\geq 0.  
\]
This leads to $\partial_r R_e(s^\ddagger,2)=0$ so that $R_e\equiv 0$ once again.  
Therefore we conclude that $R_e\equiv 0$ in every case.  
By \er{R_i1}, we also have $\rho_i\equiv 0$ and thus $V\equiv0$.
Hence $(\rho_i,R_e,V)(s^\ddagger)=(0,0,0)$ is the trivial solution.
So (1) and (2) in the theorem are valid.

%Continuing to assume that $s^\ddagger<\infty$, we now know that $\rho_i$, $R_e$ and $V$ are 
%identically zero at $s=s^\ddagger$.  
%We define $V_c^\ddagger = L\ \la(s^\ddagger)$.  
%By the simple bifurcation theorem of \cite{CR1}, the nullspace 
%$\sN = N[\partial_{(\rho_i,R_e,V)} \sF(\la(s^\ddagger),0,0,0)]$ is non-trivial 
%because the curve $\sK$ crosses the trivial curve transversely at $s=s^\ddagger$.   
%So by Lemma \ref{null1}, we have $D(V_c^\ddagger)=0$.   
%It remains to prove (3) and (4)  and also that $V_c^\ddagger > V_c^\dagger$ 
%and  $g(V_c^\ddagger) \le \frac{\pi^2}{L^2}$ .  
%
%
%Suppose on the contrary that $g(V_c^\ddagger) > \frac{\pi^2}{L^2}$.   
%Then as in the proof of Lemma \ref{null1}, the nullspace $\sN$ has a basis $(\varphi_i,\varphi_e,\varphi_v)$ 
%with 
%\[
%{\varphi}_e(x) = \sin\sqrt{g({V}_c^\ddagger)} x,
%\quad
%\sqrt{g({V}_c^\ddagger)} >\frac{\pi}{L}.
%\]
%In that case the function $\varphi_e$ has a node (changes its sign) in the interval $I$.  
%Therefore $R_e(s,\cdot)$ also has a node for $s$ near  $s^\ddagger$, Theorem
%which contradicts the positivity.  
%Thus $g(V_c^\ddagger) \le \frac{\pi^2}{L^2}$ so that the basis of $\sN$ is positive,  
%due to  Lemma \ref{null1}.   Thus (3) and (4) are valid.  

In order to complete the proof of (ii), we define $\la^\ddagger = \la(s^\ddagger)$.  
Of course, $\la^\ddagger \ge \la^\dagger$.  We must show that $\la^\ddagger > \la^\dagger$.  
On the contrary, suppose they were equal. 
Then the curve $\sK$ would go from the point $P=(\lambda^\dagger,0,0,0)$ at $s=0$ 
to the same point $P$ at $s={s^\ddagger}$.   
However by (C3) and (C4) of Theorem \ref{Global1},   
$\sK$ is a simple curve at $P$ and is real-analytic. 
So the only way $\sK$ could go from $P$ to $P$ would be if it were a loop with the part with 
$s$ approaching $s^\ddagger$ from below coinciding 
with the part with $s$ approaching $0$ from below ($s<0$).  
By (C2) of Theorem \ref{Global1}, $\rho_i(s,\cdot)$ and $R_e(s,\cdot)$ would be negative 
for $-1\ll s-s^\ddagger<0$, which would contradict their positivity.  
%Hence $\lambda^\ddagger \neq \lambda^\dagger$.
%From this and the definition of $\lambda^\dagger$ in \eqref{spv1}, 
We conclude that $\lambda^\dagger < \lambda^\ddagger$.
\end{proof}

All that remains is to prove \eqref{glow1}, namely that   
either $\sup\rho_i$ or $\sup\rho_e$ or $\la$ is unbounded as $s\to\infty$.  
Our strategy is to prove that each alternatives (a), (b) and (d) in (C5) of Theorem \ref{Global1}.  
is either impossible or implies alternative (c).  We will do this in several steps.  
In these proofs, we will use the written boundary condition from $\sF_{4}=0$ and \er{R_i1}: 
\begin{align}
\partial_rR_e(2)&=-\left(\frac{\lambda}{2}\partial_{r} H (2) + \partial_{r} V(2) \right) R_{e}(2) 
 + \frac14 \gamma  e^{\frac{\lambda}{2}H(2)}  \int_{1}^2 r^{2} h(|\partial_r \Phi (r)|)e^{-\frac{\lambda}{2}H(r)}R_{e}(r)\,dr.
\lb{bc0}
\end{align}
%and the elementary Poincar\'e inequality
%\begin{equation}\lb{Poincare}
%\|u\|_{L^2} \leq  \|\partial_r u\|_{L^2}
%\quad \text{for $u \in \{f \in H^1(I); \ f(0)=0\}$}.
%\end{equation}

%%%%%%%%%%%%  LEMMA 5.2  %%%%%%%%%%%%%%%
\begin{lem}\lb{lem(a)}
Assume the global positivity alternative in Theorem \ref{Global2}.
If $\varliminf_{s \to \infty}\la(s)= 0$, then 
$\sup_{s>0}\|V(s)\|_{H^2}$ is unbounded.
\end{lem}

\begin{proof}
On the contrary suppose that $\sup_{s>0}\|V(s)\|_{H^2}$ is bounded.
Because $\varliminf_{s \to \infty}\la(s)= 0$ and $\partial_r \Phi (s,r)>0$ 
due to the definition of $\cO$ and \eqref{H1}, 
there exists a sequence $\{s_n\}_{n \in \mathbb N}$ 
and limits $(0,V^*)$ such that 
\begin{gather*}
\left\{
\begin{array}{lllll}
 \lambda(s_n) & \to & 0 & \text{in} & \mathbb R,
 \\
 V(s_n) & \to & V^* & \text{in} & C^1(\bar{I})   \supset H^2(I),  
 \\
\partial_r \Phi (s_{n}) & \to &   \partial_{r }V^*  & \text{in} & C^0(\bar{I}),
 \end{array}\right.
%\lb{converge4}
\\
V^*=0 \ \ \ \text{ at } r=1,2,  
%\lb{boundary2}
\\
\partial_r V^*\geq 0.
%\lb{p2}
\end{gather*}
These properties clearly imply that $V^* \equiv 0$. 
It follows that
for suitably large $n$ the three  expressions $\|h(|\partial_r \Phi (s_{n})|)\|_{C^0}$, 
$|\la(s_n)|$ and 
$\|V(s_n)\|_{C^1}$, 
are arbitrarily small.  

Multiplying $\sF_2(\la(s_n),\rho_i(s_n),R_e(s_n),V(s_n))=0$ by $r^{2}R_e(s_n)$ leads to
\begin{align*}
 r^{2} (\partial_{r} R_e )^{2}(s_{n})&
 =\partial_r\left\{ r^{2} R_e(s_n) \partial_r R_e(s_n) + r^{2} \partial_r V(s_n)R_e^2(s_n) \right\}
-r^{2}\partial_r V(s_n)R_e(s_n)\partial_r R_e(s_n)
\\
&\quad -r^{2}\left\{
\frac{\lambda(s_{n})}{2} \partial_{r} V(s_{n}) \partial_{r} H
+\frac{\lambda^2(s_{n})}{4}|\partial_{r} H|^{2}
-h\left(\left| \partial_r \Phi (s_{n}) \right|\right)\right\}R_e^2(s_n).
\end{align*}
Then integrating this equation by parts over $\bar{I}$, using $R_e(s_n,0)=0$, 
and substituting $\partial_r R_e(s_n,2)$ by \er{bc0}, we  have 
\begin{align*}
&\int_1^2 r^{2}(\partial_r R_e)^2(s_n) \,dr
\\
&=-4\left(\frac{\lambda(s_{n})}{2}\partial_{r} H (2) + \partial_{r} V(s_{n},2) \right) R_{e}^{2}(s_{n},2) 
\\
&\quad  +  4 \gamma  e^{\frac{\lambda(s_{n})}{2}H(2)} R_{e}(s_{n},2)  \int_{1}^2 r^{2} h(|\partial_r \Phi (s_{n})|)e^{-\frac{\lambda(s_{n})}{2}H(r)}R_{e}(s_{n})\,dr
\\
&\quad + 4 \partial_r V(s_n,2)R_e^2(s_n,2)-\int_1^2
r^{2}\partial_r V(s_n)R_e(s_n)\partial_xR_e(s_n)\,dr
\\
&\quad -\int_1^2
r^{2}\left\{
\frac{\lambda(s_{n})}{2} \partial_{r} V(s_{n}) \partial_{r} H
+\frac{\lambda^2(s_{n})}{4}|\partial_{r} H|^{2}
-h\left(\left| \partial_r \Phi (s_{n}) \right|\right)\right\}R_e^2(s_n)
\,dr
\\
&\leq \frac{1}{2}\int_1^2 (\partial_r R_e)^2(s_n) \,dr,
\end{align*}
where for the inequality we also have used 
Sobolev's and Poincar\'e's inequalities 
and taken $n$ suitably large in deriving the last inequality.
Hence $\partial_r R_e(s_n) \equiv0$.  
Since $R_e$ vanishes at $r=1$, we conclude that $R_e(s_n)\equiv 0$,   
which contradicts its assumed positivity. 
\end{proof}

%%%%%%%%%%%%%  LEMMA 5.3  %%%%%%%%%%%%%%%%%
\begin{lem}\lb{lem(b)}
Assume the global positivity alternative in Theorem \ref{Global2}.
If $\varliminf_{s \to \infty}\inf_{r \in I} \partial_r \Phi (s_{n},r)=0$, 
then $\sup_{s>0}\{\|\rho_i(s)\|_{C^0}+\|R_e(s)\|_{H^{2}}+\|V(s)\|_{H^2}+\lambda(s)\}$ is also unbounded.
\end{lem}
\begin{proof}
On the contrary, suppose that $\sup_{s>0}\{\|\rho_i(s)\|_{C^0}+\|R_e(s)\|_{H^2}+\|V(s)\|_{H^2}+\lambda(s)\}$ is bounded.
We see from $\varliminf_{s \to \infty}\inf_{r \in I} \partial_r \Phi (s,r)=0$ 
that there exist a sequence $\{s_n\}_{n \in \mathbb N}$
and a quadruple $(\la^*,\rho_i^*,R_e^*,V^*)$ with $\la^*<\infty $ such that
\begin{gather}
\left\{
\begin{array}{llllll}
 \lambda(s_n) & \to & \la^* & \text{in} & \mathbb R,
 \\
 \rho_{i}(s_n) & \rightharpoonup & \rho_i^* & \text{in} & L^\infty(I) & \text{weakly-star},
 \\
 R_{e}(s_n) & \to & R_{e}^* & \text{in} & C^1(\bar{I}),
 \\
 \partial_r^2 R_{e}(s_n) & \rightharpoonup & \partial_r^2R_{e}^* & \text{in} & L^2(I) & \text{weakly},
 \\
 V(s_n) & \to & V^* & \text{in} & C^1(\bar{I}),
 \\
 \partial_r^2 V(s_n) & \rightharpoonup & \partial_r^2V^* & \text{in} & L^2(I) & \text{weakly},
\\
\partial_r \Phi (s_{n}) & \to & \partial_r \Phi^{*} := \partial_{r} V^* + \lambda^{*} \partial_{r}H & \text{in} & C^0(\bar{I}),
 \end{array}\right.
\lb{converge3}\\
R_{e}^*(1)=V^*(1)=V^*(2)=0,
\lb{boundary1} \\
\rho_i^* \geq 0, \quad R_e^* \geq 0,
\lb{p1}\\
\inf_{r\in \bar{I}} \partial_r \Phi^{*} (r)=0.
\lb{singular1}
\end{gather}
			We shall  show that 
\begin{equation*}
\sF_j(\lambda^*,\rho_i^*,R_e^*,V^*)=0 \quad \text{for a.e. $r$ and  $j=1,2,3$.}
\end{equation*}
For $j=1$, the equation $\sF_1(\lambda(s_n),\rho_{i}(s_n),R_{e}(s_n),V(s_n))=0$ 
with $\rho_{i}(s_n,1)=0$ is equivalent to
\begin{align*}
r^{2} \partial_r \Phi (s_{n},r) \rho_{i}(s_{n},r)
=\frac{k_e}{k_i}\int_{1}^r t^{2} h( |\partial_{r} \Phi (s_{n},r)| )e^{-\frac{\lambda(s_n)}{2}H(t)}R_{e}(s_n,t)\,dt.
\end{align*}
		Multiplying  
by a test function $\varphi \in C^0(\bar{I})$ and integrating over $\bar{I}$, we obtain
\begin{align}
\int_{1}^{2} r^{2} \partial_r \Phi (s_{n},r) \rho_{i}(s_{n},r) \varphi \,dr
=\int_{1}^{2} \frac{k_e}{k_i} \left( \int_{1}^r t^{2} h(|\partial_r \Phi (s_{n},r)|)e^{-\frac{\lambda(s_n)}{2}H(t)}R_{e}(s_n,t)\,dt \right) \varphi  \,dr.
\lb{rho0}
\end{align}
		We note that 
\begin{align*}
&{}
\left|\int_{1}^{2} r^{2}\left[  \partial_r \Phi (s_{n},r) \rho_{i}(s_{n},r)
- \partial_r \Phi^{*} (r) \rho_{i}^{*}(r)\right]\varphi \,dr\right|
\\
&\leq 
\left|\int_{1}^{2}r^{2} \left\{\partial_r \Phi (s_{n},r) - \partial_r \Phi^{*} (r) \right\} \rho_{i}(s_n,r) \varphi \,dr\right|
+\left|\int_{1}^{2} r^{2} \partial_r \Phi^{*} (r)  \{\rho_i(s_n,r) - \rho_i^*(r)\}\varphi \,dr\right|.
\end{align*}
		So passing to the limit $n\to\infty$ in \er{rho0} and making use of  \er{converge3}, we obtain 
\begin{align*}
&\int_{1}^{2} r^{2} \partial_r \Phi^{*} (r) \rho_{i}^{*}(r) \varphi  \,dr
=\int_{1}^{2} \frac{k_e}{k_i} \left( \int_{1}^r t^{2} h(|\partial_r \Phi^{*} (t)|)e^{-\frac{\lambda^{*}}{2}H(t)}R_{e}^{*}(t)\,dt \right) \varphi  \,dr,   
 \quad   \forall\varphi \in C^0(\bar{I}). 
 \end{align*} 
This immediately yields 
\begin{equation}\lb{rho1}
r^{2} \partial_r \Phi^{*} (r)\rho_{i}^{*}(r) 
=\frac{k_e}{k_i}
\int_{1}^r t^{2} h( |\partial_r \Phi^{*} (r)|)e^{-\frac{\lambda^{*}}{2}H(t)}R_{e}^{*}(t)\,dt \quad a.e.,
\end{equation}
which is equivalent to $\sF_1(\lambda^*,\rho_i^*,R_e^*,V^*)=0$ a.e.

For $j=2$, we write  $r^{2} \sF_2(\lambda(s_n),\rho_{i}(s_n),R_{e}(s_n),V(s_n))=0$ and
$R_{e}(s_n,0)=0$ in the weak form as 
\bqn    \label{weakF2}
\int_1^2 r^{2} \partial_r R_{e}(s_n)\partial_r \varphi \,dr
+\frac{\lambda^{2}(s_n)}{4}\int_1^2 r^{2} |\partial_{r} H|^{2} R_{e}(s_n)\varphi \,dr
=-\int_1^2 r^{2} G_{2n}\varphi \,dr,  \quad \forall \varphi \in H^1_0(I),
\eqn 
where 
\begin{align*}
G_{2n} &:= -\partial_r V(s_n) \partial_r R_{e}(s_n) -  \frac{1}{r^{2}}\partial_{r} ( r^{2} \partial_r V (s_n)) R_{e}(s_n)
\\
& \qquad  +\left\{\frac{\lambda(s_n)}{2}\partial_r V(s_n) \partial_r H
-h\left(|\partial_r \Phi (s_{n},r)|\right)\right\}R_{e}(s_n).
\end{align*}
Note that the second term in $r^{2}G_{2n}$ can be treated as 
\begin{align*}
&{}
\left|\int_1^2 \left[ \{ \partial_{r} ( r^{2} \partial_r V (s_n)) \}R_{e}(s_n)- \{\partial_{r} ( r^{2} \partial_r V^{*}) \} R_{e}^* \right] \varphi \,dr\right|
\\
&\leq 
\left|\int_1^2 \{\partial_{r} ( r^{2} \partial_r V (s_n)) \} (R_{e}(s_n)-R_{e}^*) \varphi \,dx\right|
+\left|\int_1^2 \{\partial_{r} ( r^{2} \partial_r V (s_n)) - \partial_{r} ( r^{2} \partial_r V^{*}) \} {R_{e}^*\varphi} \,dx\right|.   
\end{align*}
Taking the limit $n\to\infty$  of \er{weakF2} in the weak form, and using \er{converge3}, we therefore have
\[
\int_1^2 r^{2} (\partial_r R_{e}^*) \partial_r \varphi \,dr
+\frac{(\lambda^{*})^2}{4}\int_1^2 r^{2} |\partial_{r} H|^{2}  R_{e}^*\varphi \,dr
=-\int_1^2 r^{2} G_{2}^*\varphi \,dr, \quad \forall \varphi \in H^1_0(I),
\]
where 
\[
G_{2}^*:= -(\partial_r V^{*}) \partial_r R_{e}^{*}
\\
-  \frac{1}{r^{2}}\partial_{r} ( r^{2} \partial_r V^{*})R_{e}^{*}
 +\left\{
  \frac{\lambda^{*}}{2}  \partial_r V^{*} \partial_r H  
-h\left(|\partial_r \Phi^{*} |\right)\right\}R_{e}^{*} \in L^2(I).
\]
This and \er{converge3} mean that $R_e^* \in C^1(\bar{I})\cap H^{2}(I)$ 
satisfies $\sF_2=0$.
Similarly for $j=3$, we have $\sF_3(\lambda^*,\rho_i^*,R_e^*,V^*)=0$.

We now set
\[
 r_*:=\inf\{r \in \bar{I} \, ; \, \partial_r \Phi^{*} (r)=0\}.
\]
We divide our proof into two cases $r_* = 1$ and $2\ge r_* >1$. 

\smallskip

We first consider the case $r_* >1$. 
The equation \er{rho1}, which holds for a sequence $r_\nu\to r_*$ and hence for $r=r_*$, 
yields the inequality  
\begin{equation*}
0=(r_{*})^{2}\partial_r \Phi^{*}(r_{*}) \|\rho_i\|_{L^\infty(I)}  \ge 
\frac {k_e}{k_i} 
\int_{1}^{r_{*}} t^{2} h( | \partial_r \Phi^{*}  |)e^{-\frac{\lambda^{*}}{2}H}R_{e}^{*}\,dt.
\end{equation*}
Together with the nonnegativity \er{p1} of $R_e^*$,   this implies that 
$(h(| \partial_r \Phi^{*}  |)e^{-\frac{\lambda^*}{2}H}R_{e}^*)(r)=0$ 
for $r \in [1,r_*]$.
From the definition of $r_*$ and $\partial_r \Phi^{*} (r)\ge0$ in $I$,  
we see that 
\begin{equation}\lb{regular1}
\partial_r \Phi^{*} (r)>0 \quad \text{for $r\in [1,r_*)$,}
\end{equation} 
so that $h(| \partial_r \Phi^{*} |)>0$ on $[1,r_*)$.
Therefore, $R_e^*(r)\equiv0$  in $[1,r_*)$.
Hence from \er{rho1} and \er{regular1},  
$\rho_i^*=0$ a.e.  in $[1,r_*)$.
Now from the equation $\sF_3(\lambda^*,\rho_i^*,R_e^*,V^*)=0$ and $\partial_{r} (r^{2} \partial_{r} H)=0$,
we see that $r^{2}\partial_r \Phi^{*} $ is a constant in $[1,r_*]$. 
From this fact and $\partial_r \Phi^{*} (r_{*})=0$, it holds that
$\partial_r \Phi^{*} =0$ in $[1,r_*]$.
This contradicts the definition of $r_*$.

\smallskip

We now consider the other case  $r_* = 1$.  
We claim that $\pa_r\Phi^*\equiv 0$.  
On the contrary,    
suppose that there exists $t_0>0$ such that
$\pa_r\Phi^*(t_0)>0$.   Recall that $\Phi^*(1)=0$.  
Let us set
\[
 t^*:=\sup\{r<t_0 \, ; \, \partial_r \Phi^{*} (r)=0 \}.
\]
Then  $t^* \in [1,t_0)$ and $\partial_r \Phi^{*} (t^*)=0$.
Integrating $\sF_1(\lambda^*,\rho_i^*,R_e^*,V^*)=0\ a.e.$ 
over $[t^*,t]$ for any $t \in [t^*,t_0]$
and using $\sF_3(\lambda^*,\rho_i^*,R_e^*,V^*)=0$, we have
\begin{align}
&t^{4}\partial_r \Phi^{*} (t)\left[ \frac{1}{t^{2}} \left\{ \partial_{r} (r^{2} \partial_{r}  V^*) \right\}(t)+e^{-\frac{\lambda^*}{2}H}R_{e}^*(t)\right]
\notag\\
&= \frac{k_e}{k_i} t^{2}
\int_{t^{*}}^t \tau^{2} h( |\partial_r \Phi^{*} |)e^{-\frac{\lambda^{*}}{2}H}R_{e}^{*}\,d\tau
\quad \text{for a.e. $t \in [t^*,t_0]$}. 
\lb{es1}
\end{align}
By \er{p1}, \er{singular1}, and $\partial_{r} (r^{2} \partial_{r} H)=0$, 
the left hand side of \er{es1} is estimated from below as
\begin{align*}
&t^{4}\partial_r \Phi^{*} (t)\left[ \frac{1}{t^{2}} \left\{ \partial_{r} (r^{2} \partial_{r}  V^*) \right\}(t)+e^{-\frac{\lambda^*}{2}H}R_{e}^*(t)\right]
\\
&\geq t^{2} \partial_r \Phi^{*} (t) \partial_{r} \{ r^{2}  \partial_r \Phi^{*} \}(t)
=\frac{1}{2}\partial_r \left[\left\{r^{2}  \partial_r \Phi^{*} \right\}^2\right](t) \quad a.e.,
\end{align*}
since $\partial_r V^*$ is absolutely continuous.  
The integrand on the right hand side of \eqref{es1} is estimated from above by 
$C e^{-b|\partial_r \Phi^{*} |^{-1}}$, due to the behavior of $h$; see \eqref{gh}.  
Consequently, substituting these expressions into \er{es1}, 
integrating the result over $[t^*,r]$, 
and using $\partial_r \Phi^{*} (t^*)=0$, we have
\begin{equation}\lb{es2}
\left(\partial_r \Phi^{*} \right)^2(r)
\leq C \int_{t^*}^r\int_{t^*}^t e^{-b|\partial_r \Phi^{*} |^{-1}} \,d\tau dt
\quad \text{for $r \in [t^*,t_0]$}. 
\end{equation}
Now let us define $r_n$ by
\[
 r_n:=\inf\left\{r\leq t_0 \, ; \,  \partial_r \Phi^{*} = \frac{1}{n}\right\}.
\]
Notice that $t^* < r_n$ and $\partial_r \Phi^{*}  (r)\leq 1/n$ for any $r\in [t^*,r_n]$,
since the continuous function $\partial_r \Phi^{*} $ vanishes at $r=t^*$.
Then we evaluate \er{es2} at $r=r_n$ to obtain
\[
\frac{1}{n^2}
\leq C \int_{t^*}^{r_{n}}\int_{t^*}^t e^{-b|\partial_r \Phi^{*} |^{-1}} \,d\tau dt
\leq C e^{-{bn}}.
\]
For suitably large $n$, this clearly does not hold.
So once again we have  a contradiction.  

Our conclusion is that $r_*=1$ and $\partial_r \Phi^{*} \equiv 0$.
Thus we have  $\partial_{r} (r^{2} \partial_{r} V^{*}) \equiv 0$ and so 
the equation  $\sF_2(\lambda^*,\rho_i^*,R_e^*,V^*)  =0$ yields
$\partial_{r} (r^{2} \partial_r (e^{-\la^* H/2}R_e^*) ) = 0$.
This means that $e^{-\la^* H/2}R_e^*(r)=cr^{-1}+d$ for some constants $c$ and $d$.
Furthermore, $d=-c$ also follows from \er{boundary1} and $H(1)=0$.
On the other hand, \er{bc0} holds for any $s_n>0$ and 
then using \er{converge3} and $\partial_r \Phi^{*}  \equiv 0$, we have
\begin{align*}
\partial_rR_e^{*}(2)= \frac{\lambda^{*}}{2}\partial_{r} H (2) R_{e}^{*}(2).
\end{align*}
Substituting $R_e^*(r)=c(r^{-1}-1)e^{\la^* H(r)/2}$, we find $c=0$.  
Consequently, $R_e^* \equiv 0$.  
Then  
we obtain $\rho_i^* \equiv 0$ from $\sF_3(\lambda^*,\rho_i^*,R_e^*,V^*)=0$ .
Solving $\sF_3(\lambda^*,\rho_i^*,R_e^*,V^*)=0$ with \er{boundary1} 
and $\rho_i^* \equiv R_e^* \equiv 0$,
we also have $V^* \equiv 0$.
Consequently $\la^*=0$ holds and  $\varliminf_{s\to0}\la(s)=0$.
This contradicts Lemma \ref{lem(a)}, since
$\sup_{s>0}\|V(s)\|_{H^2}$ is bounded.
\end{proof}

Next, we reduce Condition (c) in Theorem \ref{Global1} to a simpler condition without derivatives.
We write the result directly in terms of the ion density $\rho_i$ and
the electron density $\rho_e=R_e e^{-\la H/2}$.

%%%%%%%%%%%%%%%%  LEMMA 5.4  %%%%%%%%%%%%%%
\begin{lem}\lb{lem(d)}
Assume the global positivity alternative in Theorem \ref{Global2}. \  
If $\sup_{s>0}\{\|\rho_i(s)\|_{C^0}+\|\rho_e(s)\|_{C^0}+\lambda(s)\}$ 
is bounded, 
then $\sup_{s>0}\{\|\rho_i(s)\|_{C^1}+\|R_e(s)\|_{C^2}+\|V(s)\|_{C^3}\}$ 
is bounded.
\end{lem}
\begin{proof}
It is clear from $\sF_3=0$
together with the definition  $\rho_e=R_ee^{-\la H/2}$
that \begin{equation*}%\lb{bound_v1}
%\sup_{s>0}\|V(s)\|_{H^2} \leq C 
\sup_{s>0}\|V(s)\|_{C^2}
\leq C \sup_{s>0}\{\|\rho_i(s)\|_{C^0}+\|\rho_e(s)\|_{C^0}\}
<+\infty. 
\end{equation*}
From this, the equation $\sF_2=0$, and $\sup_{s>0}\la(s)<+\infty$, 
we also deduce that $\sup_{s>0}\|R_e(s)\|_{C^{2}}<+\infty$.
Note that $\|\cdot \|_{H^{k}} \leq C \|\cdot \|_{C^{k}} $ for $k \in \mathbb N$.
Now  Lemma \ref{lem(b)} implies that  
$\varliminf_{s \to \infty}\inf_{r \in I}\partial_r \Phi (s,r)\neq 0$.
Together with  \er{R_i1}, this result  leads to $\sup_{s>0}\|\rho_i(s)\|_{C^1}<+\infty$.
Finally the bound $\sup_{s>0}\|\partial_r^3 V(s)\|_{C^0}$ $<+\infty$
follows from $\sF_3(\lambda(s),\rho_i(s),R_e(s),V(s))=0$.
\end{proof}

%We conclude with the following {\it main result}. 
%\begin{thm} \label{mainthm}
%Assume that the sparking voltage $\lambda^{\dagger}$ exists,
%and the transversality condition \er{transversality2} holds. 
%Then one of the following two alternatives occurs:
%\begin{enumerate}[(A)]
%\item Both $\rho_i(s,r)$ and $\rho_e(s,r)=(R_ee^{-\la H/2})(s,r)$ 
%are positive  for any $s\in(0,\infty)$ and $r \in (1,2]$.   Furthermore, 
%$\varlimsup_{s\to\infty}\{\|\rho_i(s)\|_{C^0}+\|\rho_e(s)\|_{C^0}+\lambda(s)\}=\infty$.
%\item
%There exists a finite s-value $s^\ddagger>0$ and a voltage $\lambda^\ddagger>\lambda^\dagger$ such that 
%\begin{enumerate}[(1)]  
%\item Both $\rho_i(s,r)$ and $\rho_e(s,r)$ are positive for any $s\in(0,s^\ddagger)$ and $r \in (1,2]$;
%\item $(\la(s^\ddagger),\rho_i(s^\ddagger),%\rho_e(s^\ddagger),V(s^\ddagger))=(\lambda^\ddagger,0,0,0)$.
%\end{enumerate}
%\end{enumerate}
%\end{thm}

\noindent  {\it Completion of proof of Theorem \ref{Global2}}:    
    
We have already proven that there are two alternatives in Theorem \ref{Global2}.
Assuming the first one, namely that 
both $\rho_i(s,r)$ and $\rho_e(s,r)=(R_ee^{-\la H/2})(s,r)$ 
are positive  for all $s\in(0,\infty)$ and $r \in (1,2]$, 
it suffices to prove  \eqref{glow1}, namely that 
%$\varlimsup_{s\to\infty}\{\|\rho_i(s)\|_{C^0}+\|\rho_e(s)\|_{C^0}+\lambda(s)\}=\infty$. 
either $\sup\rho_i$ or $\sup\rho_e$ or $\la$ is unbounded as $s\to\infty$.

%Suppose that (B), which  is the same as the second alternative $(ii)$ in Theorem \ref{Global2},  does {\it not} hold.  We will prove (A).  
%Then the first alternative $(i)$ in Theorem \ref{Global2} must hold.  
Looking back at Theorem \ref{Global1}, we see that there are four alternatives to the bifurcating curve $\cK$.   
Alternative (d) cannot happen because 
$\rho_i$ and $R_e$ are negative on part of the loop. 
Lemmas \ref{lem(a)} and \ref{lem(b)} assert that either (a) or (b) implies that 
$\sup_{s>0}\{\|\rho_i(s)\|_{C^0}+\|R_e(s)\|_{H^2}+\|V(s)\|_{H^2}+\lambda(s)\}$ is unbounded.
Of course,  $\|\cdot \|_{H^{k}} \leq C \|\cdot \|_{C^{k}} $.  
Thus Lemma \ref{lem(d)} implies that $\sup_{s>0}\{\|\rho_i(s)\|_{C^0}+\|\rho_e(s)\|_{C^0}+\lambda(s)\}$ 
must also be unbounded.  
Thus \eqref{glow1} holds.
\qed 

%This concludes the proof of Theorem \ref{mainthm0}.  
%{\  We remark that (B) never occurs 
%unless a voltage $V_c^\ddagger > V_c^\dagger$ exists satisfying \er{positive1}. }

%%%%%%%%%%%%%

\section{Asymptotically High Voltage} \label{SAHV}

In this section we consider the case when both densities remain positive all along the whole global bifurcating curve $\cK$.  We will prove that if both densities are bounded along a sequence $s_n\to\infty$, 
then they actually tend to zero.  This means the gas becomes dilute.  By Theorem \ref{Global2}, in this scenario the voltages along the sequence must be unbounded.  

%It is of interest to know how the global bifurcation curve behaves 
%for the case that the densities are bounded but $\lambda$ is unbounded.
%%%%%%%%%%%%%%%%  LEMMA 6.1  %%%%%%%%%%%%%%%%%%%
\begin{pro}\lb{lem(c)}

Assume $\gamma(1+\gamma)^{-1}\neq e^{-a}$ and Alternative {\rm (i)} in Theorem \ref{Global2}.
If $\varlimsup_{s\to\infty}\{\|\rho_i(s)\|_{C^0}+\|\rho_e(s)\|_{C^0}\}$ is bounded, 
then $\varliminf_{s\to\infty}\{\|\rho_i(s)\|_{C^0} + \|\rho_e(s)\|_{L^{1}}\}=0$ 
and $\varlimsup_{s\to\infty} \lambda(s) = \infty$.
%Furthermore, suppose that there exists a sequence $\{s_n\}_{n \in \mathbb N}$ such that
%\begin{equation}\lb{cond5}
%\lim_{n\to\infty}s_n=\infty, \quad
%\sup_{n \geq 1}(\|\rho_i(s_n)\|_{C^0}+\|\rho_e(s_n)\|_{C^0})<+\infty, \quad
%\lim_{n\to\infty}\la(s_n)=\infty.
%\end{equation}
%Then 
%$\lim_{{n}\to\infty} (\|\rho_i(s_{n})\|_{C^0}+\|\rho_e(s_{n})\|_{L^1})=0$. 
\end{pro}

\begin{proof}  
We see from \er{sp0} that $(\rho_i,\rho_e,V)$ solves the system 
\begin{subequations}\lb{r5}
\begin{gather}
k_i \partial_{r} \left\{ r^{2} \rho_i  \left(\partial_{r} V+ \la\partial_{r} H \right) \right\}
=k_e r^{2} h\left(\left| \partial_{r} V+ \la\partial_{r} H \right|\right)\rho_e,
\lb{re6} \\
-  \partial_{r} \{r^{2} \partial_{r}  \rho_{e}+ r^{2} ( \partial_{r}V+\lambda \partial_{r} H) \rho_{e}  \}
=r^{2} h\left(\left| \partial_{r} V+ \la\partial_{r} H \right|\right)\rho_e,
\lb{re7} \\
\partial_{r} ( r^{2} \partial_{r}  V) = r^{2}(\rho_i-\rho_e),
\lb{re8} \\ 
\partial_r\rho_e(2) +  (\partial_{r} V+ \la\partial_{r} H) \rho_e (2)  =
\gamma\frac{k_i}{k_e} (\partial_{r} V+ \la\partial_{r} H) \rho_i (2)
\lb{re9}
\end{gather}
with boundary conditions
\begin{equation}\lb{rb5}
\rho_i(1)=\rho_e(1)=V(1)=V(2)=0.
\end{equation}
\end{subequations}
Obviously, $\varlimsup_{s\to\infty} \lambda(s) = \infty$ holds from the condition $\varlimsup_{s\to\infty}\{\|\rho_i(s)\|_{C^0}+\|\rho_e(s)\|_{C^0}\}<+\infty$
and Alternative (i) in Theorem \ref{Global2}.
Therefore, supposing that  
there is a sequence $\{s_n\}_{n \in \mathbb N}$ such that
\begin{equation}\lb{cond5}
\lim_{n\to\infty}s_n=\infty, \quad
\sup_{n \geq 1}(\|\rho_i(s_n)\|_{C^0}+\|\rho_e(s_n)\|_{C^0})<+\infty, \quad
\lim_{n\to\infty}\la(s_n)=\infty , 
\end{equation}
it suffices to show $\lim_{{n}\to\infty} (\|\rho_i(s_{n})\|_{C^0}+\|\rho_e(s_{n})\|_{L^1})=0$.

First, it is clear from \er{re8} and \er{rb5} that
\begin{equation}\lb{bound_V}
\sup_{n\geq1}\|V(s_n)\|_{C^2}
\leq C \sup_{n\geq1}\{\|\rho_i(s_n)\|_{C^0}+\|\rho_e(s_n)\|_{C^0}\}
<+\infty. 
\end{equation}
Solving \er{re6} for $\partial_r \rho_i$ and writing $h$ explicitly from \er{gh}, we have 
\[
\partial_r \rho_i
=\frac{k_e}{k_i} a \exp\left(\frac{-b}{|\partial_{r} V+ \la\partial_{r} H|}\right) 
\frac{|\partial_{r} V+ \la\partial_{r} H|}{r^{2}(\partial_{r} V+ \la\partial_{r} H)}\rho_e
- \frac{\partial_{r} ( r^{2} \partial_{r}  V) }{r^{2}(\partial_{r} V+ \la\partial_{r} H)}\rho_i.
\]
From 
this identity, \er{cond5}, \er{bound_V}, $\partial_{r} H \geq \tfrac{1}{2}$ 
and $\lim_{n\to\infty} \la(s_n)=\infty$, we see that
$\sup_{n\geq1}\|\rho_i(s_n)\|_{C^1}<+\infty$.  
Thus there exist a subsequence [still denoted by $s_n$]
and $(\rho_i^*,\rho_e^*,V^*)$ such that
\begin{gather}
\left\{
\begin{array}{llllll}
 \lambda(s_n) & \to & \infty & \text{in} & \mathbb R,
 \\
 \rho_{i}(s_n) & \to & \rho_i^* & \text{in} & C^0(\bar{I}), & 
 \\
 \rho_{e}(s_n) & \rightharpoonup & \rho_{e}^* & \text{in} & L^\infty(I) & \text{weakly-star},
 \\
 V(s_n) & \to & V^* & \text{in} & C^1(\bar{I}),
 \end{array}\right.
\lb{converge5}\\
\rho_{1}^*(1)=V^*(1)=V^*(2)=0,
\lb{boundary5} \\
\rho_i^* \geq 0, \quad \rho_e^* \geq 0.
%\lb{p5}
\notag
\end{gather}  

Now 
integrating \er{re7} over $[r,2]$, using \er{re9} and $h(\ell)=a\ell\exp(-\frac{b}{\ell})$,  
and multiplying the result by $\lambda^{-1}$, we obtain
		\begin{align}    \label{identityG}
& 
-4\gamma\frac{k_i}{k_e}\left(\frac{\partial_r V(2)}{\lambda}+\partial_{r} H(2) \right)\rho_i(2)
+ \frac{r^{2}}{\la} \partial_{r}  \rho_{e}(r)
+r^{2} \left( \frac{\partial_{r}V(r)}{\la}+ \partial_{r} H(r)\right) \rho_{e}(r)
\\    \label{identityH}  
&=a\int^2_r t^{2} \exp\left(\frac{-b}{|\partial_{r} V(t)+ \la\partial_{r} H(t)|}\right)
\left| \frac{\partial_{r} V(t)}{\la}+ \partial_{r} H(t)\right| \rho_e(t)\,dt.
		\end{align}
We take this identity at $s=s_n$ and will look at the behavior of each term as $s_n\to\infty$.  
We multiply it by a test function $\phi \in C_0^\infty(I)$, 
integrate it over $\bar{I}$, and let $n \to \infty$.
		From \er{cond5}--\er{converge5}, each term in \er{identityG} converges as follows:
\begin{gather*}
-4\gamma\frac{k_i}{k_e}\left(\frac{\partial_r V(s_n,2)}{\lambda(s_n)}+\partial_r H(2)\right)\rho_i(s_n,2)\int_1^2   \phi(r) \,dr
\to - 4\gamma\frac{k_i}{k_e} \partial_r H (2) \rho_i^*(2) \int_1^2  \phi(r) \,dr,
\\
\int_1^2 \frac{r^{2}}{\lambda}\partial_r\rho_e(s_n,r) \phi(r) \,dr
=-\int_1^2 \frac{1}{\lambda}\rho_e(s_n,r) \{2r \phi(r) + r^{2}\partial_r\phi(r) \} \,dr
\to 0,
\\
\int_1^2 r^{2} \left(\frac{\partial_x V(s_n,r)}{\lambda(s_n)}+\partial_r H(r)\right)\rho_e(s_n,r)  \phi(r) \,dr
\to \int_1^2 r^{2} \partial_r H(r) \rho_e^*(r)\phi(r) \,dr.
\end{gather*}
		In order to investigate 
\er{identityH} multiplied by $\phi(r)$ and integrated, we let
\begin{align*}
I_{0,n}&:=\int_1^2 \left[a\int^2_r t^{2} \exp\left(\frac{-b}{|\partial_r V(s_n,t)+\lambda(s_n)\partial_{r} H (t)|}\right)
\left|\frac{\partial_r V(s_n,t)}{\lambda(s_n)}+\partial_{r} H (t) \right| \rho_e(s_n,t)\,dt \right] \phi(r)\,dr,
\\
I_{1,n}&:=\int_1^2 \left[ a \int^2_r t^{2} \partial_{r} H (t) \rho_e(s_n,t)  \,dt \right]\phi(r)\,dr.
\end{align*}
		Then from \er{cond5}--\er{converge5} and $\partial_{r} H \geq \tfrac{1}{2}$ we have 
\begin{align*}
&\begin{aligned}
|I_{0,n}-I_{1,n}| &\leq       \|\rho_e\|_{L^1}\|\phi\|_{L^1} 
\\
& \quad \times \sup_{t \in \bar{I}}  \left| \exp\left(\frac{-b}{|\partial_r V(s_n,t)+\lambda(s_n)\partial_{r} H (t)|}\right)
\left|\frac{\partial_r V(s_n,t)}{\lambda(s_n)}+\partial_{r} H (t) \right|-\partial_{r} H (t)\right| \to 0,
\end{aligned}
\\
&\begin{aligned}
I_{1,n} & = a \int_1^2 t^{2} \partial_{r} H (t) \rho_e(s_n,t)  \left[\int_1^t \phi(r) \,dr \right]\,dt
\\
& \to \int_1^2 \left[ a \int^2_r t^{2} \partial_{r} H (t) \rho^{*}_e(t)  \,dt \right]\phi(r)\,dr.
\end{aligned}
\end{align*}
		Therefore we deduce from \er{identityG}--\er{identityH} that
\begin{equation}\lb{rhoe*}
- 4\gamma\frac{k_i}{k_e} \partial_r H (2) \rho_i^*(2) 
 + r^{2} \partial_r H(r) \rho_e^*(r)
 =  a \int^2_r t^{2} \partial_{r} H (t) \rho^{*}_e(t)  \,dt  \quad a.e. 
\end{equation}

Now integrating \er{re6} with \er{boundary5}, we have
\begin{align}   
\rho_i(s_n,r) 
=a r^{-2}\left\{\frac{\partial_{r} V (s_n,r)}{\la(s_{n})}+ \partial_{r} H (r) \right\}^{-1}
\frac{k_e}{k_i}\int_1^r   K_n(t)   \rho_e(s_n,t) \,dt
\leq C \|\rho_e(s_n)\|_{L^1},
\label{Kny} 
\end{align}
				where 
\bqn  
K_n(t)   := t^{2} \exp\left(\frac{-b}{|\partial_r V(s_n,t)+\lambda(s_n)\partial_{r} H (t)|}\right)
\left|\frac{\partial_r V(s_n,t)}{\lambda(s_n)}+\partial_{r} H (t)\right|.
\notag
\eqn 	
We used \er{bound_V} in deriving the last inequality.
Evaluating \er{Kny} at $r=2$, we have 
\begin{align*}
4\left\{\frac{\partial_{r} V (s_n,2)}{\la(s_{n})}+ \partial_{r} H (2) \right\}\rho_i(s_n,2) 
=a \frac{k_e}{k_i}\int_1^2   K_n(t)   \rho_e(s_n,t) \,dt.
\end{align*}
But $K_n(t)\to t^{2} \partial_{r} H (t)$ uniformly and $\rho_e(s_n) \rightharpoonup \rho_e^*$ in $L^\infty$ weakly-star.  
Therefore, letting $n\to\infty$, 
we get 
\begin{equation}\lb{rhoi*}
4 \partial_r H (2) \rho_i^*(2) =a\frac{k_e}{k_i}\int_1^2 t^{2} \partial_{r} H (t) \rho_{e}^{*}(t)  \,dt. 
\end{equation}
 Combining \er{rhoe*} and \er{rhoi*} we find the identity 
\begin{equation}\lb{claim*}
-a\gamma \int_1^2 t^{2} \partial_{r} H (t) \rho_{e}^{*}(t)  \,dt
 + r^{2} \partial_r H(r) \rho_e^*(r)
 =  a \int^2_r t^{2} \partial_{r} H (t) \rho^{*}_e(t)  \,dt  \quad a.e. 
\end{equation}
		Note that 
\er{claim*} implies that $(\cdot)^{2} \partial_{r} H (\cdot) \rho^{*}_e(\cdot) $ is a continuous function. 
Now multiplying the identity by $e^{ar}$, we have
\begin{equation*}
\partial_r \left(e^{ar} \int_r^2 t^{2} \partial_{r} H (t) \rho^{*}_e(t)   \,dt\right)=-a\gamma e^{ar} \int_1^2 t^{2} \partial_{r} H (t) \rho^{*}_e(t)   \,dt
\quad a.e. 
\end{equation*}
Then integration  over $\bar{I}$ leads to
\begin{equation*}
\int_1^2 t^{2} \partial_{r} H (t) \rho^{*}_e(t)  \,dt \, \{e^{a}-\gamma(e^{2a}-e^{a})\}=0,
\end{equation*}
which together with $\partial_{r} H \geq \tfrac12$ and the assumption $\gamma(1+\gamma)^{-1}\neq e^{-a}$ 
means that $\|\rho_e^*\|_{L^1}=0$.  
We also see from \er{converge5} and $\rho_e(s_n)\geq 0$ that
\begin{equation*}%\lb{converge6}
\|\rho_e(s_n)\|_{L^1}=\int_1^2 1\cdot \rho_e(s_n,r)\,dr \to 
\int_1^2 1\cdot \rho_e^*(r)\,dr=0 \quad \text{as $n\to\infty$}. 
\end{equation*}
It follows that $\|\rho_e(s_n)\|_{L^1} \to 0$  {\it for the whole original sequence}.
Together with \er{Kny} we also get $\rho_i(s_n)\to 0$ uniformly in $\bar I$.  
This completes the proof of the proposition.
\end{proof}
It seems possible that $\rho_e$ also tends to zero in the uniform norm but we have not succeeded to prove it.  

%\begin{rem}\blue
%It seems impossible to prove that $\rho_e$ tends to zero uniformly.
%For very large $\lambda$, let us simplify the ODE \eqref{re7} by replacing $\partial_{r} V+ \la\partial_{r} H$ by $\lambda$ and $h(|\partial_{r} V+ \la\partial_{r} H|)$ also by $\lambda$ as $\tfrac{1}{\lambda}\rho_e'' +  \rho_e' +  \rho_e = 0$. Then the solution with the condition $\rho_{e}(1)=0$ is 
%\begin{gather*}
%\rho_{e}=C_{1} e^{\frac{-1-\sqrt{1-\frac4\lambda}}{2} r} - C_{1} e^{{-\sqrt{1-\frac4\lambda}}} e^{\frac{-1+\sqrt{1-\frac4\lambda}}{2} r},
%\end{gather*}
%where $C_{1}$ is a constant. It converges to $C_{1} (e^{-r} - e^{-1})$ as $\lambda \to \infty$.
%\end{rem}

%%%%%%%%%%%%%%%%%%%%%%%%%%%%%%%
%%%%%%%%%%%%%%%%%%%%%%%%%%%%%%%
\section{Sparking Voltage with Dominant Cathode Emission}\label{SVDCE}

This section is devoted to proving Theorem \ref{SparkingVoltage}. 
Note that $g(\frac{2\la}{r^2})  = a\frac{2\la}{r^2} e^{-\frac{br^2}{2\la}} - \frac{\la^2}{r^4}$ is negative in $\Gamma=\left\{ (a, b, \gamma) \, | \,  \gamma>\tfrac 1a,\ \ b>\tfrac 4ea\right\}$.
Indeed, the condition $b>\frac4ea$ implies that 
\begin{gather}\label{g>0}
g(s)=a\ell e^{-\frac b \ell}-\frac{\ell^2}{4} = \ell^{2}\left( \frac{a}{\ell e^{\frac b \ell}}-\frac14\right) < 0  \quad  \text{for $\ell>0$},
\end{gather}
where the inequality follows directly from the fact that the minimum of $\ell e^{\frac b \ell}$ occurs at $\ell=b$ with the value $eb$.  

Recall that the solution $\varphi_{e}$ of the electron system \eqref{eqU} is unique up to a constant factor if it exists.  
Because $g<0$, any non-trivial solution must satisfy $\varphi_{e}'(2)\neq 0$. 
(Indeed, if $\varphi_{e}'(2)=0$, then $\varphi_{e}\equiv0$,  
as we see by multiplying the equation \eqref{eqUa} by $r^{2} \varphi_{e}$ and integrating.)    
%to $\varphi_{e}(r)>0$ for $r \in I$ by multiplying the equation \eqref{eqUa} by $r^{2} \varphi_{e}$, integrating the resultant equality over $I$, and using the condition $\varphi_{e}(1)=0$.
Therefore,  normalizing $\varphi_{e}$ by $\varphi_{e}'(2)$, we see that the electron system \eqref{eqU} is equivalent to the following problem:
\begin{subequations}\label{equ1}
\begin{gather}
u'' + \frac 2ru' - \frac{\la^2}{r^4}u  + \frac{2a}{r^2}\la e^{-br^2/2\la}u =
u'' + \frac 2ru' +g\left(\frac{2\la}{r^2}\right) u =0,  
\label{equ1a}\\
u(1)=0,\quad u'(2)=1, %\frac{1}{4} + \frac{1}{4}e^{-\la}.
\label{equ1b}\\
 u(r)>0\ \ \text{ in }\ I,
\label{equ1c} \\
B(\la,u) =  u'(2) + \tfrac\lambda{4} u(2) - \gamma \int_1^2  p(r,\lambda) u(r) dr=0, 
\quad \text{where }\ 
p(r,\la) = a\tfrac\la 2 e^{-\frac\la 2} e^ {\frac\la r} e^{-\frac{br^2}{2\la}}.
\label{equ1d}
\end{gather}
\end{subequations}
%For brevity we may rewrite \eqref{eqUc} as  
%\[  
%B(\la,u) =  u'(2) + \tfrac\lambda{4} u(2) - \gamma \int_1^2  p(r,\lambda) u(r) dr, 
%\quad \text{ where }\ 
%p(r,\la) = a\tfrac\la 2 e^{-\frac\la 2} e^ {\frac\la r} e^{-\frac{br^2}{2\la}}. 
%\]  
Now we are in a position to prove that the sparking parameter set $\cS$ contains $\Gamma$, as stated in Theorem \ref{SparkingVoltage}.
\begin{lem}\label{PhiinA}
For every $(a, b, \gamma)\in \Gamma$, there exists a pair $(\lambda,\varphi_{e}) \in \mathbb R_{+} \times H^2(I)$ that solves  the electron system  \eqref{eqU}.
\end{lem}

\begin{proof}    %[Proof of Theorem \ref{SparkingVoltage}]
%It suffices to find a pair $(\lambda,u) \in \mathbb R_{+} \times H^2(I)$ that solves  
We will work with the equivalent system  \eqref{equ1}.
For any $\la>0$, we denote any solution of the boundary value problem \eqref{equ1a}--\eqref{equ1b} by $u=u(r)=u(r,\la)$.
%Now for any $\la>0$ we define the function $u=u(r)=u(r,\la)$ that solves the boundary value problem 
%\begin{subequations}\label{equ1}
%\begin{gather}
%u'' + \frac 2ru' - \frac{\la^2}{r^4}u  + \frac{2a}{r^2}\la e^{-br^2/2\la}u =
%u'' + \frac 2ru' +g\left(\frac{2\la}{r^2}\right) u =0,  \\
%u(1)=0,\quad u'(2)=\frac{1}{4} + \frac{1}{4}e^{-\la}.
%\end{gather}
%\end{subequations}
%%\footnote{For the case $\la \gg1$, we can show that $g <0$. But we need to define $U$ for all $\la >0$.}
%%Under the assumption $\max_{r \in [1,2], \ \lambda>0 }  g  < 0$, 
Because $g<0$ as we showed in \eqref{g>0}, 
the spectrum of the operator $-\Delta-g(\cdot)$ with the boundary conditions 
$u(1)=0$ and $u'(2)=0$ belongs entirely to the positive real axis $(0,\infty)$.  
Hence for any $\la>0$ the solution $u(r)$ of \eqref{equ1a}--\eqref{equ1b} exists and is unique.

We claim that the solution $u$ of the boundary value problem \eqref{equ1a}--\eqref{equ1b} also satisfies \eqref{equ1c}, that is, $u$ is positive on $I$ for all $\la>0$.
On the contrary, suppose that $u(s) \leq 0$ holds for some $s \in I$.
We first treat the case $u(s) < 0$.
Due to $u'(2)>0$, $u$ attains a minimum at some $s_{0} \in I$.
Therefore, $u(s_{0})<0$, $u'(s_{0})=0$, and  $u''(s_{0}) \geq 0$.
Evaluating the equation at $r=s_{0}$, we have a contradiction
\begin{gather*}
0 \leq u''(s_{0})  =-g u(s_{0}) <0.
\end{gather*}
We next treat the possibility that $u(s) = 0$.
There are two possibilities $u\equiv 0$ on $[1,s]$ and $u\not\equiv 0$ on $[1,s]$.
In the former case, we arrive at the contradiction $u\equiv 0$ on $I$.
In the latter case, there is either a maximum at $s_{0} \in (1,s)$ or a minimum at $s_{0} \in (1,s)$.
If we have a minimum, we arrive at a contradiction as we did above.
If we have a maximum, then $u(s_{0})>0$, $u'(s_{0})=0$, and  $u''(s_{0}) \leq 0$ hold.
Evaluating the equation at $r=s_{0}$, we have the contradiction
\begin{gather*}
0 \geq u''(s_{0})  =-g u(s_{0}) >0.
\end{gather*}
Thus $u > 0$ holds on $I$.

It remains to find $\lambda>0$ so that the solution $u=u(r,\lambda)$ of the boundary value problem \eqref{equ1a}--\eqref{equ1b} satisfies \eqref{equ1d}. 
For this purpose, we define the auxiliary function $U$ for any $\lambda>0$ by
\[
U = U(r,\la): = \alpha_{\lambda} \left\{ \frac{1}{\la} e^{\frac{1}{2} \la}e^{-\la/r} -  \frac{1}{\la} e^{-\frac{3}{2}\la}e^{\la/r}  \right\}, \quad  \text{where \ $\alpha_{\lambda}:=\left(\frac{1}{4} + \frac{1}{4}e^{-\la}\right)^{-1}$.}
 \]
It is positive in $(1,2]$ and is a solution of the simple boundary value problem   
\[  
U''+\frac2r U' - \frac{\la^2}{r^4}U=0, 
\quad  U(1)=0,\quad  U'(2)=1.  
\]
We claim that $B(\la,U)<0$ for $\la \gg 1$. 
Indeed, direct computation shows that  
\begin{align*}    
B(\la,U)  &=  U'(2) + \frac\la{4} U(2) - \gamma \int_1^2  p(r,\la)\cdot  U(r) dr
\\
& = \alpha_{\lambda} \left\{ \frac{1}{4}  + \frac{1}{4} e^{- \la} +\frac{1}{4}  - \frac{1}{4} e^{- \la}
-\frac12 \gamma a e^{-2\la}  \int_1^2  e^{\frac\la r - \frac{br^2}{2\la}} \cdot \left( e^{(2-\frac1r)\la} -e^{\frac{\lambda}{r}} \right)dr
\right\}
\\
& = \alpha_{\lambda} \left\{ \frac{1}{2} -\frac12  \gamma a e^{-2\la }  \int_1^2  e^{\frac\la r - \frac{br^2}{2\la}} \cdot \left( e^{(2-\frac1r)\la} -e^{\frac{\lambda}{r}} \right)dr   
\right\}.
\end{align*}
We observe that one integral term is 
\begin{align*}
-\frac12  \gamma ae^{-2\la}  \int_1^2  e^{\frac\la r - \frac{br^2}{2\la}} \cdot e^{(2-\frac1r)\la} dr
&= -\frac12  \gamma a     \int_1^2 e^{-\frac{br^2}{2\la}} dr 
\\
& \leq  -\frac12 \gamma a    \int_1^2 \left( 1 - \frac{br^2}{2\la} \right) dr 
= -\frac12 \gamma a    \left(1-\frac{7b}{6\la}\right).  
\end{align*}  
The other integral term is estimated as 
\begin{align*}
\frac12 \gamma a e^{-2\la }  
& \int_1^2  e^{\frac\la r - \frac{br^2}{2\la}} \cdot e^{\frac{\lambda}{r}} dr \leq \frac12  \gamma a  e^{-2\la  }    \int_1^2  e^{2\frac\la r }dr
\\
& \leq  \frac12  \gamma a  e^{-2\la  }   
 \left(  \int_1^{1+\delta}  e^{2\la}dr + \int_{1+\delta}^2  e^{\frac {2\la}{1+\delta} }dr  \right)
= \frac12 \gamma a \left(\delta + (1-\delta) e^{-2\la\delta/(1+\delta)  }\right).
\end{align*}
Putting these estimates into the expression for $B(\la,u)$, we see that
\begin{align*}
B(\lambda,U) & \leq  \frac{\alpha_{\lambda}}{2}   \left\{ 1  - \ga a \left[1 -\tfrac{7b}{6\la} 
-\delta - (1-\delta) e^{-2\la \delta/(1+\delta)} \right] \right\}
\notag \\ 
&=  \frac{\alpha_{\lambda}}2\left\{  1-\ga a + \ga a\delta + O\left(\frac1\la\right)  \right\}.  
\end{align*}
Now taking $\delta$ to be small enough and $\la$ to be large  enough and using  the assumption $\gamma a> 1$, 
we obtain 
\begin{gather*}
B(\la,U) < 0 \quad \text{for $\la\gg 1$}.  
\end{gather*}
This proves the claim.  
In particular,  $B(\lambda,U)=O(1)$ for $\la \gg 1$ owing to $2 \leq  \alpha_{\lambda} \leq 4$.

Next, we show that $u$ is near $U$ for large $\la$, in the sense that 
\begin{gather}\label{H11}
\|(u-U)'\|_{L^2}^{2} + \lambda^{2} \|u-U\|_{L^2}^{2} \leq \frac{C}{\la^{2}}  \quad \text{for $\la\gg 1$},
\end{gather}
where $C$ is a constant independent of $\la$.
Indeed, it is clear that the difference $V:=u-U$ solves
\begin{gather}
V'' + \frac 2r V' - \frac{\la^2}{r^4}V  + \frac{2a}{r^2}\la e^{-br^2/2\la} V = -\frac{2a}{r^2}\la e^{-br^2/2\la}U,
\label{Veq1}\\
V(1)=0,\quad V'(2)=0.
\label{Vbc1}
\end{gather}
Recall the weighted inner product defined in \eqref{inner-product1}.
We multiply \eqref{Veq1} by $r^{2}V$, integrate it over $\bar{I}$, and use \eqref{Vbc1} to obtain
\begin{gather*}
\| V'\|_{L^{2}}^{2}+\la^{2} \| r^{-2} V\|^{2}_{L^{2}} \leq C \la \|V\|_{L^{2}}^{2} + C\la \|U\|_{L^{2}}\|V\|_{L^{2}}.
\end{gather*}
Using this and $ \|U\|_{L^{2}}=O(\la^{-1})$, and considering only sufficiently large $\la$, we arrive at \eqref{H11}.

Now we claim that $u\ge U\ge0$.  
Indeed, this is the same as $V\ge0$.  Supposing on the contrary that  
%Now we verify that $B(u,\la)<0$.
%It follows from the definition of $U$ that $U \geq 0$.
%We prove that $u \geq U$, i.e. $V \geq 0$.
 $V(s) < 0$  for some $s \in (1,2]$, 
then $V$ would attain its minimum at $s_{0} \in (1,2]$.
Therefore $V(s_{0})<0$, $V'(s_{0})=0$, and  $ V''(s_{0}) \geq 0$.
Evaluating the differential equation at $r=s_{0}$ and using $g<0$, we have the contradiction
\begin{gather*}
0 \leq  V''(s_{0})  =-g V(s_{0})  -\frac{2a}{s_{0}^2}\la e^{-b s_{0}^2/2\la}U(s_{0}) <0 . 
\end{gather*}
Thus $u \geq U$.  

We next observe that 
\begin{align*}
B(\la,u) &:= u'(2) + \frac\lambda{4} u(2) - \gamma \int_1^2  p(r,\lambda) u(r) dr
\\
&=U'(2) + \frac\lambda{4} U(2) - \gamma \int_1^2  p(r,\lambda) U(r) dr + \frac\lambda{4}  (u(2)-U(2)) - \gamma \int_1^2  p(r,\lambda) (u-U)(r) dr
\\
&\leq U'(2) + \frac\lambda{4} U(2) - \gamma \int_1^2  p(r,\lambda) U(r) dr 
+ \frac\lambda{4}  V(2) 
\\
&\leq  B(\lambda,U)+ C\lambda^{-\frac12},
\end{align*}  
because 
$(V(2))^2 \le C\|V\|_{L^2}  \|V'\|_{L^2}  \le C\frac1\la \frac1{\la^2} = C\la^{-3}$,   
where we have used   \eqref{H11} for the last inequality. 
Since $B(\la,U)<0$ for $\la\gg 1$, we must also have $B(\la,u)<0$ for $\la\gg1$.  

On the other hand, consider the simple case $\la=0$.  For that case,  $u(r,0) = 2-\frac2r$, so that 
$B(0,u) = u'(2,0)=\frac12>0$.  
Thus $B$ is positive at $\la=0$ and negative at sufficiently large $\la$.  
Moreover $B(\la,u(\cdot,\la))$ is a continuous function of $\la$ for $0\le\la<\infty$.  
We infer that there exists $\la_1>0$ such that 
$B(\la_1,u(\cdot,\la_1))=0$.  
%where $u_{1}$ solves \eqref{equ1} with $\la=\la_{1}$.
Thus $u(\cdot,\la_1)$ satisfies \eqref{equ1} and \er{Veq1}, 
which means that $(a,b,\ga)\in\cS$. 
This is exactly what we wanted to prove.     \end{proof}

Recall that the electron system \eqref{eqU} is equivalent to \eqref{equ1}. 
Furthermore, Lemmas \ref{Null1} and \ref{codim1} ensure that ${\rm dim} \, N(\sL)={\rm dim} \, N(\cL^{*})=1$ for $\la=\la^\dagger$, and also that there exists a function $u^{\dagger}$ solving \eqref{equ1} with $\la=\la^\dagger$. The solution $\psi_{e}$ of \eqref{PhiEq1} is one-dimensional. 
Then we see that $\psi_{e}(2)\neq 0$.
Indeed, if $\psi_{e}(2)=0$ holds, $\psi_{e}\equiv0$ follows from \eqref{PhiEq1} and $g<0$, and thus it contradicts to the fact that the solution is one-dimensional.
It is clear that $w:=\frac{\psi_{e}}{\psi_{e}(2)}$ solves the boundary value problem
\begin{subequations}              \label{wEq1}
\begin{gather}
-\tfrac{1}{r^{2}} \partial_{r} ( r^{2} \partial_{r} w )-g(\lambda \partial_{r} H)w
-h(\lambda\partial_{r} H)e^{-\frac{\lambda}{2}H} {\gamma} e^{\frac{\la}{2}}=0,
\label{wEq2}\\
w(1)=\partial_{r} w(2) + \tfrac{\lambda}{4} =0.
%- \gamma \frac{\lambda}{2} e^{-\frac{\lambda}{2}} \int_{1}^{2} a e^{-\frac{b r^{2}}{2 \lambda}}e^{\frac{\lambda}{r}}S_e(r) \,dr=0, 
\end{gather}
\end{subequations}
Due to $g<0$ in $\Gamma$, the spectrum of the operator $-\Delta-g(\cdot)$ with the boundary conditions 
$w(1)=0$ and $\partial_{r} w(2)=0$ belongs entirely to the positive real axis $(0,\infty)$.  
Hence for any $\la>0$ the solution $w(r,\la)$ of \eqref{wEq1} exists and is unique.  
Now the transversality condition $F(a, b, \gamma) \neq 0$ at the sparking voltage $\lambda^{\dagger}$ can be written as $\widetilde{F}(a,b,\gamma) \neq 0$, where
\begin{align*}    
\widetilde{F}(a,b,\gamma)&:=-{\gamma} e^{\frac{\la^\dagger}{2}}
\int_1^2 r^{2}  \left\{ h'\left( \la^\dagger \partial_{r} H \right) (\partial_{r} H) - h\left( \la^\dagger \partial_{r} H \right)\frac{H}{2} \right\}e^{-\frac{\la^\dagger}{2}H} u^{\dagger} \,dr 
   \notag\\
&\quad -\int_1^2 r^{2} w g'(\la^\dagger \partial_{r} H) (\partial_{r} H) u^{\dagger} \, dr+        
u^{\dagger}(2) - 2\partial_r u^{\dagger} (2)  - \frac{\lambda}  {2} u^{\dagger} (2).
%\label{transversality5}
\end{align*}

\begin{lem}
For almost every $(a, b, \gamma)\in \Gamma$, the transversality condition $F(a, b, \gamma) \neq 0$ holds at the sparking voltage $\lambda^{\dagger}=\lambda^{\dagger}(a,b,\gamma)$.
\end{lem}
\begin{proof}
It suffices to show $\widetilde{F}(a,b,\gamma) \neq 0$ for almost every $(a, b, \gamma)\in \Gamma$.
Let us denote the unique solution of  \eqref{equ1a}--\eqref{equ1b} by $u(r,\la; a,b)$. It is an analytic function for $(\la;a,b) \in \mathbb R_{+}^{3}$ by analyticity of solutions to BVPs with respect to parameters.  
Then we see directly from \er{eqUc} that $B(u,\la;a,b,\ga)$ is an analytic function of all its variables $(\la;a,b,\ga)$.  
Let us also denote the unique solution of \er{wEq1} by $w(r,\la; a,b,\gamma)$. It is an analytic function for $(\la;a,b,\gamma) \in \mathbb R_{+}^{4}$.

Clearly the set $\Gamma$ is open in $\real^3$.  Note that $B(\la^\dagger,a,b,\ga)=0$. 
Now let 
$$ Z = \left\{(a,b,\gamma)\in \Gamma \, ; \, %g(V_c^\dagger,a,b)\ne0,\ 
\frac{\pa B}{\pa \lambda}(\lambda^\dagger,a,b,\gamma)=0 \right\},$$  
where $\lambda^\dagger = \lambda^\dagger(a,b,\gamma)$.  
We claim that $\Gamma\cap Z^{c}$ is an open set.  
For each point $(a,b,\gamma)\in \Gamma \cap Z^c$, 
we can apply the real-analytic version of the implicit function theorem 
to the equation $B(\lambda^\dagger,a,b,\gamma)=0$.
There is a neighborhood of $(a,b,\gamma)$
in which the function $\lambda^\dagger$ is real-analytic and
$\frac{\pa B}{\pa \lambda}(\lambda^\dagger,a,b,\gamma)\ne0$.
Thus $\Gamma\cap Z^c$ is open.  
Furthermore, $\lambda^\dagger: \Gamma \cap Z^c\to\real$ is a real-analytic function for which  
$\frac{\pa B}{\pa \lambda}(\lambda^\dagger,a,b,\gamma)$ does not vanish.

Next we claim that the $Z$ also has $\real^3$-measure zero.
Within $Z$ both of the equations, $B=0$ and $\frac{\pa B}{\pa \lambda}=0$, are satisfied by $(\lambda^\dagger, a,b,\gamma)$.  
Using these equations, we have
\begin{align}
\gamma &= \left\{u'(2,\lambda^\dagger;a,b)  + \frac{\lambda^{\dagger}}{4} u(2,\lambda^\dagger;a,b) \right\} \left( \int_1^2  p(r,\lambda^\dagger;a,b) u(r,\lambda^\dagger;a,b) dr \right)^{-1}, 
%\quad \text{ where }\ 
%p(r,\la) = a\tfrac\la 2 e^{-\frac\la 2} e^ {\frac\la r} e^{-\frac{br^2}{2\la}}
\label{B=0}\\  
0   &=   \frac{\pa u'}{\pa \lambda}(2,\lambda^\dagger;a,b) 
+ \frac{\lambda^{\dagger}}{4}  \frac{\pa u}{\pa \lambda}(2,\lambda^\dagger;a,b)
+ \frac{1}{4}u(2,\lambda^\dagger;a,b)
\notag \\
& \quad - \gamma \int_1^2   \frac{\pa p}{\pa \lambda}(r,\lambda^\dagger;a,b) u(r,\lambda^\dagger;a,b) dr
- \gamma \int_1^2  p(r,\lambda^\dagger;a,b)  \frac{\pa u}{\pa \lambda}(r,\lambda^\dagger;a,b) dr.
\label{B'=0}
\end{align} 
We plug \eqref{B=0} into $\gamma$ in \eqref{B'=0}. Then the resultant equation contains no explicit $\gamma$, and is a single equation for $(\lambda^\dagger,a,b)$.  
Thus, within $Z$, the function $\lambda^\dagger$ depends only on $(a,b)$.  
Hence, using \er{B=0} within $Z$, we see that the variable $\gamma$ is determined uniquely by $(a,b)$.    
So, due to the Fubini--Tonelli theorem, $Z$ has $\real^3$-measure zero.

Recall that $\lambda^\dagger: \Gamma \cap Z^c\to\real$ is real-analytic.
It follows that the set  $\{(a,b,\gamma)\in \Gamma \cap Z^c \, ; \, \widetilde{F}(a,b,\gamma)=0\}$ also has measure zero 
because the zero set of any analytic function $\not\equiv0$ must have measure zero.  
The proof is complete.
\end{proof}

\end{document}